\documentclass[onecolumn,11pt]{article}
\usepackage[top=1in, bottom=1in, left=1in, right=1in]{geometry}
\usepackage{amsmath,amsfonts,amscd,amssymb,amsthm}
\usepackage{fancyhdr}
\usepackage{enumitem}
\usepackage[percent]{overpic}
\usepackage{tikz}
\usetikzlibrary{shapes.misc}
\usepackage{mathrsfs}
\usepackage{ifthen,tabularx,graphicx,multirow}
\usepackage{xargs}
\usepackage[colorinlistoftodos,prependcaption,textsize=tiny]{todonotes}
\usepackage{tcolorbox}
\usepackage{enumerate}
\usepackage{comment}
\usepackage{algorithm}
\usepackage{booktabs}
\usepackage{varwidth}
\usepackage{graphicx}
\usepackage{epstopdf}
\usepackage{mathtools}
\usepackage{setspace}
\usepackage{palatino} 
\usepackage{bm}
\usepackage[bottom,flushmargin,hang,multiple]{footmisc}

\usepackage[pdfencoding=auto, psdextra]{hyperref}
\usepackage{algpseudocode}
\usepackage[capitalise,noabbrev]{cleveref}
\usepackage{lscape}
\usepackage{afterpage}
\usepackage{multirow}

\usepackage{arydshln}

\crefformat{equation}{(#2#1#3)}

\makeatletter
\renewcommand\paragraph{\@startsection{paragraph}{4}{\z@}%
            {-2.5ex\@plus -1ex \@minus -.25ex}%
            {1.25ex \@plus .25ex}%
            {\normalfont\normalsize\bfseries}}
\makeatother

\makeatletter
\def\hlinewd#1{%
\noalign{\ifnum0=`}\fi\hrule \@height #1 \futurelet
\reserved@a\@xhline}
\makeatother

\tikzset{cross/.style={cross out, draw=blue, minimum size=2*(#1-\pgflinewidth), inner sep=0pt, outer sep=0pt},
	cross/.default={1pt}}

\newcommand\blfootnote[1]{%
  \begingroup
  \renewcommand\thefootnote{}\footnote{#1}%
  \addtocounter{footnote}{-1}%
  \endgroup
}

\newcommand*{\dd}{{\,\mathrm{d}}}
\newcommand{\Av}{\mathbf{A}}

\newcommand{\Cv}{\mathbf{C}}

\newcommand{\Fv}{\mathbf{F}}
\newcommand{\Gv}{\mathbf{G}}

\newcommand{\Iv}{\mathbf{I}}

\newcommand{\Kv}{\mathbf{K}}
\newcommand{\Lv}{\mathbf{L}}
\newcommand{\Mv}{\mathbf{M}}

\newcommand{\Qv}{\mathbf{Q}}

\newcommand{\Uv}{\mathbf{U}}
\newcommand{\Vv}{\mathbf{V}}
\newcommand{\Wv}{\mathbf{W}}
\newcommand{\Xv}{\mathbf{X}}
\newcommand{\Yv}{\mathbf{Y}}
\newcommand{\Zv}{\mathbf{Z}}
\newcommand{\Psiv}{\mathbf{\Psi}}
\newcommand{\Sigmav}{\mathbf{\Sigma}}

\newcommand{\gv}{\mathbf{g}}

\newcommand{\vv}{\mathbf{v}}
\newcommand{\wv}{\mathbf{w}}
\newcommand{\xv}{\mathbf{x}}
\newcommand{\yv}{\mathbf{y}}

\numberwithin{equation}{section}

\newtheorem{proposition}{Proposition}

\numberwithin{theorem}{section}
\numberwithin{lemma}{section}
\numberwithin{corollary}{section}
\numberwithin{proposition}{section}
\numberwithin{definition}{section}
\numberwithin{example}{section}
\makeatletter 
\newcounter{algorithmicH}
\let\oldalgorithmic\algorithmic
\renewcommand{\algorithmic}{
  \stepcounter{algorithmicH}
  \oldalgorithmic}
\renewcommand{\theHALG@line}{ALG@line.\thealgorithmicH.\arabic{ALG@line}}
\makeatother

\title{{\fontsize{16}{16}\selectfont \textbf{Another look at Residual Dynamic Mode Decomposition in the regime of fewer Snapshots than Dictionary Size}}\vspace{-.15in}}

\author{\normalsize{Matthew J. Colbrook$^{1*}$}\\
\footnotesize{$^1$DAMTP, University of Cambridge, Cambridge, CB3 0WA, United Kingdom}}
\date{}
\begin{document}
\maketitle

\blfootnote{$^*$ Corresponding author (m.colbrook@damtp.cam.ac.uk).}
\vspace{-.2in}
\begin{abstract}
Residual Dynamic Mode Decomposition (ResDMD) offers a method for accurately computing the spectral properties of Koopman operators. It achieves this by calculating an infinite-dimensional residual from snapshot data, thus overcoming issues associated with finite truncations of Koopman operators, such as spurious eigenvalues. These spectral properties include spectra and pseudospectra, spectral measures, Koopman mode decompositions, and dictionary verification. In scenarios where there are fewer snapshots than dictionary size, particularly for exact DMD and kernelized EDMD, ResDMD has traditionally been applied by dividing snapshot data into a training set and a quadrature set. Through a novel computational approach of solving a dual least-squares problem, we demonstrate how to eliminate the need for two datasets. We provide an analysis of these new residuals for exact DMD and kernelized EDMD, demonstrating ResDMD's versatility and broad applicability across various dynamical systems, including those modeled by high-dimensional and nonlinear observables. The utility of these new residuals is showcased through three diverse examples: the analysis of cylinder wake, the study of aerofoil cascades, and the compression of data from transient shockwave experimental data. This approach not only simplifies the application of ResDMD but also extends its potential for deeper insights into the dynamics of complex systems.
\end{abstract}

\section{Introduction}

In this paper, we consider discrete-time dynamical systems:
\begin{equation}
\label{eq:DynamicalSystem} 
\xv_{n+1} = \Fv(\xv_n), \qquad n= 0,1,2,\ldots, 
\end{equation}
where $\xv\in\Omega$ denotes the state of the system, and $\Omega\subseteq\mathbb{R}^d$ is the statespace. The function $\Fv:\Omega \rightarrow \Omega$ governs the system's evolution. In many modern applications, the system's dynamics are too complicated to describe analytically, or we only have access to incomplete knowledge of $\Fv$. Often, knowledge is limited to discrete-time snapshots of the system, i.e., a finite dataset
$$
\left\{\xv^{(m)},\yv^{(m)}\right\}_{m=1}^M\quad \text{such that}\quad \yv^{(m)}=\Fv(\xv^{(m)}),\quad m=1,\ldots,M.
$$
We concisely write this data in the form of snapshot matrices
\begin{equation}
\label{eq:snapshot_data}
\Xv=\begin{pmatrix} 
\xv^{(1)} & \xv^{(2)} & \cdots & \xv^{(M)}
\end{pmatrix}\in\mathbb{R}^{d\times M},\quad \Yv=\begin{pmatrix} 
\yv^{(1)} & \yv^{(2)} & \cdots & \yv^{(M)}
\end{pmatrix}\in\mathbb{R}^{d\times M}.
\end{equation}
The question becomes how to use this data to meaningfully study
the dynamical system.

Koopman operators \cite{koopman1931hamiltonian,koopman1932dynamical} offer a powerful framework to study this question by addressing the fundamental issue of \textit{nonlinearity} of \cref{eq:DynamicalSystem}. We lift a nonlinear system \eqref{eq:DynamicalSystem} into an infinite-dimensional space of observable functions $g:\Omega\rightarrow\mathbb{C}$ using a Koopman operator $\mathcal{K}$:
$$
[\mathcal{K}g](\xv) = g(\Fv(\xv)), \quad \text{so that}\quad [\mathcal{K}g](\xv_n)=g(\xv_{n+1}).
$$
Through this approach, the evolution dynamics become linear, enabling the use of generic solution techniques based on spectral decompositions. Hence, the fundamental objective is determining the spectral properties of $\mathcal{K}$ from the snapshot data in \cref{eq:snapshot_data}. Since $\mathcal{K}$ acts on an \textit{infinite-dimensional} function space, we have exchanged the nonlinearity in \eqref{eq:DynamicalSystem} for an infinite-dimensional spectral problem associated with the unitary operator $\mathcal{K}$. In particular, the spectral properties of $\mathcal{K}$ can be significantly more complex and challenging to compute than those of a finite matrix \cite{colbrookthesis,SCI_ref,colbrook2022computation,colbrook2022foundations,colbrook2021computingCIMP,colbrook2019compute}.

One of the most successful algorithms for spectral analysis of Koopman operators is \textit{Dynamic Mode Decomposition} (DMD), initially developed in the context of fluid dynamics \cite{schmid2009dynamic,schmid2010dynamic}. The \textit{Koopman mode decomposition} (KMD) was introduced in \cite{mezic2005spectral}, providing a theoretical basis to connect DMD with Koopman operators \cite{rowley2009spectral}. This connection validated DMD's application in data-driven dynamical systems and offered a powerful yet straightforward approach for approximating the spectral properties of Koopman operators. Fusing contemporary Koopman theory with an efficient numerical algorithm has led to significant advancements and a surge in research. For example, the reader can consult the early reviews \cite{mezic2013analysis,budivsic2012applied} and more recent reviews \cite{brunton2021modern,colbrook2023multiverse}.

However, practitioners soon realized that simply building linear models in terms of the primitive
measured variables cannot sufficiently capture nonlinear dynamics beyond periodic and quasi-periodic phenomena. A breakthrough occurred with the introduction of Extended DMD (EDMD) \cite{williams2015data}, which generalizes DMD to a broader class of basis functions to expand the Koopman operator's approximate eigenfunctions. One may think of EDMD as a Galerkin method to approximate the Koopman operator by a finite matrix. It is well-known that truncating/discretizing an infinite-dimensional spectral problem to a finite matrix eigenvalue problem is fraught with danger. Common challenges include spurious eigenvalues, missing parts of the spectrum, instabilities caused by discretization and dealing with continuous spectra.

\begin{table}
\renewcommand{\arraystretch}{1.7}
\footnotesize
\centering
\begin{tabular}{|c|c:c|c:c|} 
\cline{2-5}
\multicolumn{1}{c|}{}          & \multicolumn{2}{c|}{\textbf{Exact DMD}} & \multicolumn{2}{c|}{\textbf{Kernelized EDMD}}\\ 
\hline
\!\textbf{Truncated SVD}\! & \multicolumn{2}{c|}{$\Xv \approx \Uv \mathbf{\Sigma}\Vv^*$} & \multicolumn{2}{c|}{$\sqrt{\Wv}\Psiv_X \approx \Qv\Sigmav \Zv^*$}\\ 
\hline
\multirow{2}{*}{\!\textbf{Least-Squares}\!} & \!\! Orig.: \! & $\|\Yv{-}\Kv\Xv\|_\mathrm{F}$ &\!\! Orig.:\! & $\|\sqrt{\Wv}\Psiv_Y{-}\sqrt{\Wv}\Psiv_X \Kv\|_\mathrm{F}$ \\ 
\cdashline{2-5}&\!\! Dual:\! &\!\! $\|\Yv\Vv\Sigmav^{-1}{-}\Xv\Vv\Sigmav^{-1}\widetilde{\Kv}^*\|_\mathrm{F}$\!\! &\!\! Dual:\! & \!\!$\|(\sqrt{\Wv}\Psiv_Y)^*\Qv\Sigmav^{-1}{-}(\sqrt{\Wv}\Psiv_X)^*\Qv\Sigmav^{-1}\widehat{\Kv}^*\|_\mathrm{F}$\!\!\!\\ 
\hline
\!\textbf{New Residual}\!& \multicolumn{2}{c|}{\!\!$\!{\frac{1}{\|\vv\|}\sqrt{\vv^*\widetilde{\Lv}\vv{-}\lambda\vv^*\lambda\widetilde{\Kv}^*\vv{-}\overline{\lambda}\vv^*\widetilde{\Kv}\vv{+}|\lambda|^2\vv^*\vv}}\!$\!\!}& \multicolumn{2}{c|}{\!$\!{\frac{1}{\|\vv\|}\sqrt{\vv^*\widehat{\Lv}\vv{-}\lambda \vv^*\widehat{\Kv}^*\vv{-}\overline{\lambda}\vv^*\widehat{\Kv}\vv{+}|\lambda|^2\vv^*\vv}}\!$\!}  \\
\hline
\!\textbf{Additional Matrix}\!&\multicolumn{2}{c|}{\!$\widetilde{\Lv}_{\mathrm{DMD}}=(\Yv\Vv\Sigmav^{-1})^*(\Yv\Vv\Sigmav^{-1})$\!}&\multicolumn{2}{c|}{\!$\widehat{\Lv}=(\Qv\Sigmav^{-1})^*\sqrt{\Wv}\Psiv_Y\Psiv_Y^*\sqrt{\Wv}(\Qv\Sigmav^{-1})$\!}\\
\hline
\end{tabular}
\renewcommand{\arraystretch}{1}
\normalsize
\caption{High-level overview of the new residuals for exact DMD and kernelized EDMD. The idea is to consider a dual least-squares problem. The matrices and SVD for kernelized EDMD are indirectly computed using the kernel trick.}
\label{schematic}
\end{table}

Residual DMD (ResDMD) has been introduced to tackle these problems \cite{colbrook2021rigorous} (see also the applications in \cite{colbrook2023residualJFM}). The idea of ResDMD is to introduce a new (finite) matrix computed from the snapshot data that allows the computation of infinite-dimensional residuals to control truncation/discretization errors. ResDMD computes spectra and pseudospectra of general Koopman operators with error control and computes smoothed approximations of spectral measures (including continuous spectra) with explicit high-order convergence theorems. ResDMD thus provides robust and verified Koopmanism, and can also be used to verify chosen dictionaries of observables and Koopman mode decompositions. The convergence theorems in \cite{colbrook2021rigorous} (as with all convergence theorems regarding DMD and the properties of Koopman operators) are stated in terms of the so-called large data limit, where $M\rightarrow\infty$.

Often, one deals with scenarios where $M$, the number of snapshots, is bounded above by $N$, the dimension of the subspace spanned by the dictionary. In this regimen, a naïve implementation of ResDMD leads to vanishing residuals. So far, the solution in papers that use ResDMD has been to split up the snapshot data into one subset that trains or finds a dictionary and another subset that computes residuals. This approach allows the convergence theory to hold. In this paper, we propose a simple alternative strategy. We consider a dual least-squares problem associated with the regime $M\leq N$ that leads to a dual residual that measures the objective function value of the associated minimization problem. When using ResDMD with exact DMD, this residual corresponds to the right approximate eigenvectors of the DMD matrix. When using ResDMD with kernelized EDMD, this residual corresponds to the left approximate eigenvectors of the EDMD matrix. This new method completely avoids the need to split up the snapshot data and associates a natural residual with DMD methods when $M\leq N$. The idea is summarized in \cref{schematic} for exact DMD and kernelized EDMD but can be readily extended to other DMD methods.

The paper is organized as follows. In \cref{sec:prelim}, we provide preliminaries concerning EDMD and ResDMD. We then build an appropriate residual and algorithms for using ResDMD with exact DMD in \cref{sec:exact_ResDMD}. These ideas are extended to kernelized EDMD in \cref{sec:kernelized_ResDMD}. We end with three examples in \cref{sec:examples}. The first is a simple validation on the classical cylinder wake problem. The second example concerns verified Koopman modes of a periodic cascade of aerofoils. In the final example, we demonstrate how using residuals can lead to a more efficient compression on a transient shockwave with a highly non-normal Koopman operator. Code for all of these examples can be found at \url{https://github.com/MColbrook/Residual-Dynamic-Mode-Decomposition/tree/main/Examples_gallery_4}.

\section{Preliminaries}
\label{sec:prelim}

This section briefly recalls the notation and ideas of EDMD and ResDMD.

\subsection{EDMD}
\label{sec:EDMD}

The standard DMD algorithm can accurately characterize periodic and quasi-periodic behaviors in nonlinear systems. However, DMD models based on linear observables generally fail to capture truly nonlinear phenomena. \textit{Extended DMD} (EDMD) \cite{williams2015data} addresses this limitation by allowing nonlinear observables.\footnote{The connection between EDMD and the earlier variational approach of conformation dynamics \cite{noe2013variational,nuske2014variational} from molecular dynamics is explored in \cite{wu2017variational,klus2018data}.} EDMD constructs a finite-dimensional approximation of $\mathcal{K}$ utilizing the snapshot data in~\cref{eq:snapshot_data}. One first chooses a dictionary $\{\psi_1,\ldots,\psi_{N}\}$, i.e., a list of observables, in the space $L^2(\Omega,\omega)$, where $\omega$ is a positive (not necessarily finite) measure on the statespace $\Omega$. These observables generate a finite-dimensional subspace $V_N=\mathrm{span}\{\psi_1,\ldots,\psi_{N}\}$. EDMD selects a matrix $\Kv\in\mathbb{C}^{N\times N}$ with the goal of approximating $\mathcal{K}$ on $V_N$:
$$
{[\mathcal{K}\psi_j](\xv) = \psi_j(\Fv(\xv)) \approx \sum_{i=1}^{N} (\Kv)_{ij} \psi_i(\xv)},\quad1\leq j\leq {N}.
$$
Define the vector-valued feature map $\Psi(\xv)=\begin{bmatrix}\psi_1(\xv) & \cdots& \psi_{{N}}(\xv) \end{bmatrix}\in\mathbb{C}^{1\times {N}}.$ Any $g\in V_{N}$ can be written as $g(\xv)=\sum_{j=1}^{N}\psi_j(\xv)g_j=\Psi(\xv)\,\gv$ for some vector $\gv\in\mathbb{C}^{N}$. It follows that
$$
	[\mathcal{K}g](\xv)=\Psi(\Fv(\xv))\,\gv=\Psi(\xv)(\Kv\,\gv)+\underbrace{\left(\sum_{j=1}^{N}\psi_j(\Fv(\xv))g_j-\Psi(\xv)(\Kv\,\gv)\right)}_{R(\gv,\xv)}.
$$
Typically, $V_{N}$ is not an invariant subspace of $\mathcal{K}$ and no choice of $\Kv$ makes $R(\gv,\xv)$ zero for all $g\in V_N$ and $\omega$-almost every $\xv\in\Omega$. Hence, we select $\Kv$ as a solution of
\begin{equation} \label{eq:ContinuousLeastSquaresProblem}
	\min_{\Kv\in\mathbb{C}^{N\times N}} \left\{\int_\Omega \max_{\substack{\gv\in\mathbb{C}^{N}\\\|\Cv\gv\|_{\ell^2}=1}}|R(\gv,\xv)|^2\dd\omega(\xv)=\int_\Omega \left\|\Psi(\Fv(\xv))\Cv^{-1} - \Psi(\xv)\Kv\Cv^{-1}\right\|^2_{\ell^2}\dd\omega(\xv)\right\}.
\end{equation}
Here, $\Cv$ is an invertible positive self-adjoint matrix that controls the size of $g=\Psi\gv$. 

We approximate the integral in~\eqref{eq:ContinuousLeastSquaresProblem} via a quadrature rule with nodes $\{\xv^{(m)}\}_{m=1}^{M}$ and weights $\{w_m\}_{m=1}^{M}$.  For notational convenience, let $\Wv=\mathrm{diag}(w_1,\ldots,w_{M})$ and
\begin{equation*}
	\begin{split}
		\Psiv_X=\begin{pmatrix}
			       \Psi(\xv^{(1)}) \\
			       \vdots              \\
			       \Psi(\xv^{(M)})
		       \end{pmatrix}\in\mathbb{C}^{M\times N},\quad
		\Psiv_Y=\begin{pmatrix}
			       \Psi(\yv^{(1)}) \\
			       \vdots              \\
			       \Psi(\yv^{(M)})
		       \end{pmatrix}\in\mathbb{C}^{M\times N}.
	\end{split}
\end{equation*}
The discretized version of~\eqref{eq:ContinuousLeastSquaresProblem} is the following weighted least-squares problem:
\begin{equation}
	\label{EDMD_opt_prob2}
	\min_{\Kv\in\mathbb{C}^{N\times N}}\left\{\sum_{m=1}^{M} w_m\left\|\Psi(\yv^{(m)})\Cv^{-1}-\Psi(\xv^{(m)})\Kv\Cv^{-1}\right\|^2_{\ell^2}\!=\!\left\|\Wv^{1/2}\Psiv_Y\Cv^{-1}-\Wv^{1/2}\Psiv_X \Kv\Cv^{-1}\right\|_{\mathrm{F}}^2\right\},
\end{equation}
where $\|\cdot\|_{\mathrm{F}}$ denotes the Frobenius norm.
By reducing the size of the dictionary if necessary, we may assume without loss of generality that $\Wv^{1/2}\Psiv_X$ has rank $N$. A solution to~\eqref{EDMD_opt_prob2} is
\begin{equation}
\label{EDMD_matrix_def}
\Kv= (\Wv^{1/2}\Psiv_X)^\dagger \Wv^{1/2}\Psiv_Y=(\Psiv_X^*\Wv\Psiv_X)^\dagger\Psiv_X^*\Wv\Psiv_Y,
\end{equation}
where `$\dagger$' denotes the pseudoinverse. Though this solution is independent of $\Cv$, choosing the correct $\Cv$ is vital in other variants \cite{colbrook2023mpedmd}. We define the matrices
\begin{equation}
\label{eq_EDMD_corr_matrices}
\Gv=\Psiv_X^*\Wv\Psiv_X = \sum_{m=1}^{M} w_m \Psi(\xv^{(m)})^*\Psi(\xv^{(m)}),\quad \Av=\Psiv_X^*\Wv\Psiv_Y = \sum_{m=1}^{M} w_j \Psi(\xv^{(m)})^*\Psi(\yv^{(m)}).
\end{equation}
If the quadrature converges\footnote{There are typically three scenarios for convergence of the quadrature rule (see the discussion in \cite{colbrook2023multiverse}): high-order quadrature rules, suitable for small $d$ and when we are free to choose $\{\xv^{(m)}\}_{m=1}^{M}$; drawing $\{\xv^{(m)}\}_{m=1}^{M}$ from a single trajectory and setting $\omega_m=1/M$, suitable for ergodic systems; and drawing $\{\xv^{(m)}\}_{m=1}^{M}$ at random.} then
\begin{equation}
	\label{quad_convergence}
	\lim_{M\rightarrow\infty}\Gv_{jk} = \langle \psi_k,\psi_j \rangle\quad \text{ and }\quad
	\lim_{M\rightarrow\infty}\Av_{jk} = \langle \mathcal{K}\psi_k,\psi_j \rangle,
\end{equation}
where $\langle \cdot,\cdot \rangle$ is the inner product associated with $L^2(\Omega,\omega)$. It follows that $\Kv$ approaches a matrix representation of $\mathcal{P}_{V_{N}}\mathcal{K}\mathcal{P}_{V_{N}}^*$ in the large data limit $M\rightarrow\infty$. Here, $\mathcal{P}_{V_{N}}$ denotes the orthogonal projection onto $V_{N}$. Hence, in the large data limit, EDMD corresponds to the so-called finite section method, which can suffer from the issues discussed in the introduction. These issues are well-documented; for example, the discussion in \cite{colbrook2023multiverse}.

\subsection{ResDMD}

In general, the eigenvalues of the EDMD matrix $\Kv$ may be spurious due to taking a projection onto the finite-dimensional subspace $V_N$, even after taking the limit $M\rightarrow\infty$. Moreover, we may not approximate all of the spectrum of $\mathcal{K}$, and it can be hard to validate the output of EDMD. Care must be taken when discretizing or truncating an infinite-dimensional operator to a finite matrix to compute spectral properties, and this is well-documented throughout the Koopman literature.

\textit{Residual DMD} \cite{colbrook2021rigorous} overcomes these limitations by introducing a simple additional matrix that is computed from the snapshot data. The main idea behind ResDMD is to compute an infinite-dimensional residual. Consider an observable $g=\Psiv\gv\in V_N$, which we aim to be an approximate eigenfunction of $\mathcal{K}$ with an approximate eigenvalue $\lambda$. For now, the method of determining the pair $(\lambda,g)$ is left unspecified. A way to measure the suitability of the candidate pair $(\lambda,g)$ is through the relative residual
\begin{equation}
\label{residual1}
\frac{\|(\mathcal{K}-\lambda I)g\|}{\|g\|}=\sqrt{\frac{\int_{\Omega}|[\mathcal{K}g](\xv)-\lambda g(\xv)|^2\ \mathrm{d} \omega(\xv)}{\int_{\Omega}|g(\xv)|^2\ \mathrm{d} \omega(\xv)}}=\sqrt{\frac{\langle \mathcal{K}g,\mathcal{K}g \rangle-\lambda\langle g,\mathcal{K}g \rangle-\overline{\lambda}\langle \mathcal{K}g,g \rangle+|\lambda|^2\langle g,g \rangle}{\langle g,g \rangle}}.
\end{equation}
For instance, if $\mathcal{K}$ is a normal operator (one that commutes with its adjoint), then
$$
\mathrm{dist}(\lambda,\mathrm{Sp}(\mathcal{K}))=\inf_{f}\frac{\|(\mathcal{K}-\lambda I)f\|}{\|f\|}\leq\frac{\|(\mathcal{K}-\lambda I)g\|}{\|g\|}.
$$
Generalizing to also non-normal $\mathcal{K}$, the residual in \eqref{residual1} is closely related to the concept of pseudospectra \cite{trefethen2005spectra}. Adopting the quadrature interpretation of EDMD, we can define a finite data approximation of the relative residual as:
$$
\mathrm{res}(\lambda,g)=\|(\Wv^{1/2}\Psiv_Y-\lambda\Wv^{1/2}\Psiv_X)\gv\|_{\ell^2}/\|\Wv^{1/2}\Psiv_X\gv\|_{\ell^2}.
$$
We then have
\begin{align}
[\mathrm{res}(\lambda,g)]^2&=\frac{\gv^*\left[\Psiv_Y^*\Wv\Psiv_Y-\lambda \Psiv_Y^*\Wv\Psiv_X -\overline{\lambda} \Psiv_X^*\Wv\Psiv_Y + |\lambda|^2\Psiv_X^*\Wv\Psiv_X \right]\gv}{\gv^*\Psiv_X^*\Wv\Psiv_X\gv}\notag\\
&=\frac{\gv\left[\Psiv_Y^*\Wv\Psiv_Y-\lambda\Av^*-\overline{\lambda}\Av+|\lambda|^2\Gv\right]\gv}{\gv^*\Gv\gv},\label{residual2}
\end{align}
where $\Gv$ and $\Av$ are the same matrices from \eqref{eq_EDMD_corr_matrices}. The right-hand side of \eqref{residual2} has an additional matrix $\Lv=\Psiv_Y^*\Wv\Psiv_Y$. Under the assumption that the quadrature rule converges, this matrix approximates $\mathcal{K}^*\mathcal{K}$:
\begin{equation}
	\label{quad_convergence2}
	\lim_{M\rightarrow\infty}\Lv_{jk} = \langle \mathcal{K}\psi_k,\mathcal{K}\psi_j \rangle.
\end{equation}
Comparing \eqref{residual1} and the square-root of \eqref{residual2}, we observe that
$$
\lim_{M\rightarrow\infty}\mathrm{res}(\lambda,g)=\|(\mathcal{K}-\lambda I)g\|/\|g\|.
$$
Note that there is no approximation or projection on the right-hand side of this equation. Consequently, we can compute an \textit{infinite-dimensional residual} directly using finite matrices, achieving exactness in the limit of large data sets. ResDMD leverages this residual in a suite of algorithms to compute various spectral properties of $\mathcal{K}$. This paper focuses on using residuals to validate candidate eigenpairs and compute pseudospectra.

\section{ResDMD for exact DMD}
\label{sec:exact_ResDMD}

The \textit{exact DMD} algorithm \cite{tu2014dynamic} is the workhorse DMD algorithm. This section shows how to use ResDMD in tandem with exact DMD. This is a particular case of the more general kernelized methods that we treat in \cref{sec:kernelized_ResDMD}, and also serves to introduce the basic ideas.

\subsection{Exact DMD}

Given the snapshot matrices $\Xv,\Yv\in\mathbb{R}^{d\times M}$ in \eqref{eq:snapshot_data}, DMD seeks a matrix $\Kv_{\mathrm{DMD}}\in\mathbb{R}^{d\times d}$ such that $\Yv\approx \Kv_{\mathrm{DMD}}\Xv.$ We can think of this as constructing a \textit{linear and approximate} dynamical system. Notice that the DMD matrix $\Kv_{\mathrm{DMD}}$ acts on state vectors, whereas the EDMD matrix $\Kv$ in \cref{EDMD_matrix_def} acts on coefficients. This duality will be important in what follows.

To find a suitable matrix $\Kv_{\mathrm{DMD}}$, we consider the least-squares minimization problem
\begin{equation}
\label{DMD_opt_vanilla}
\min_{\Kv_{\mathrm{DMD}}\in\mathbb{R}^{d\times d}} \left\|\Yv-\Kv_{\mathrm{DMD}}\Xv\right\|_{\mathrm{F}}.
\end{equation}
A solution to the problem in \eqref{DMD_opt_vanilla} is
$$
\Kv_{\mathrm{DMD}}=\Yv \Xv^{\dagger}\in\mathbb{R}^{d\times d},
$$
where $\dagger$ denotes the Moore--Penrose pseudoinverse. Often, $d\gg M$ and we first project onto a low-dimensional subspace to reconstruct the leading nonzero eigenvalues and eigenvectors of the matrix $\Kv_{\mathrm{DMD}}$ without explicitly computing it. The exact DMD \cite{tu2014dynamic} algorithm does this using an SVD and is summarized in \cref{alg:DMD_vanilla}, where we have assumed that the projected DMD matrix is diagonalizable.\footnote{We make this assumption about various matrices throughout. If this does not hold, one could consider a Schur decomposition \cite{drmavc2023data} that provides an orthogonal set of interacting modes or a block-diagonal Schur decomposition with nearly confluent eigenvalues grouped.} The rank $r$ for the truncated singular value decomposition is usually chosen based on the decay of singular values of $\Xv$.\footnote{If low-dimensional structure is present in the data \cite{udell2019big}, the singular values decrease rapidly, and small $r$ captures the dominant modes. Moreover, the lowest energy modes may be corrupted by noise, and low-dimensional projection is a form of spectral filtering which has the positive effect of dampening the influence of noise \cite{hansen2006deblurring}.}

\begin{algorithm}[t]
\textbf{Input:} Snapshot data $\Xv\in\mathbb{R}^{d\times M}$ and $\Yv\in\mathbb{R}^{d\times M}$, rank $r\in\mathbb{N}$. \\
\vspace{-4mm}
\begin{algorithmic}[1]
\State Compute a truncated SVD of the data matrix $
\Xv \approx \Uv \mathbf{\Sigma}\Vv^*$, $\Uv\in\mathbb{R}^{d\times r}$, $\mathbf{\Sigma}\in\mathbb{R}^{r\times r}$,$\Vv\in\mathbb{R}^{M\times r}.$ The columns of $\Uv$ and $\Vv$ are orthonormal and $\mathbf{\Sigma}$ is diagonal.
\State Compute the compression $
\widetilde{\Kv}_{\mathrm{DMD}}=\Uv^*\Yv\Vv\mathbf{\Sigma}^{-1}\in\mathbb{R}^{r\times r}.$
\State Compute the eigendecomposition $
\widetilde{\Kv}_{\mathrm{DMD}}\Wv=\Wv\mathbf{\Lambda}$.

\noindent{}The columns of $\Wv$ are eigenvectors and $\mathbf{\Lambda}$ is a diagonal matrix of eigenvalues.
\State Compute the modes $
\mathbf{\Phi}=\Yv\Vv\mathbf{\Sigma}^{-1}\Wv.$
\end{algorithmic} \textbf{Output:} The eigenvalues $\mathbf{\Lambda}$ and modes $\mathbf{\Phi}\in\mathbb{R}^{d\times r}$.
\caption{The exact DMD algorithm \cite{tu2014dynamic}, which has become the workhorse DMD algorithm.}
\label{alg:DMD_vanilla}
\end{algorithm}

We can interpret exact DMD as constructing a linear model of the dynamical system on projected coordinates $\widetilde{\xv}=\Uv^*\xv$. Namely, $\widetilde{\xv}_{n+1}\approx \widetilde{\Kv}_{\mathrm{DMD}} \widetilde{\xv}_{n}$. The left singular vectors $\Uv$ are known as proper orthogonal decomposition (POD) modes \cite{berkooz1993proper}. If the SVD is exact, so that $\Xv = \Uv \mathbf{\Sigma}\Vv^*$, then
	$$
	\Kv_{\mathrm{DMD}}=\Yv \Xv^{\dagger}=\Yv\Vv\mathbf{\Sigma}^{-1}\Uv^*.
	$$
	Using this relation, we have
	$$
	\Kv_{\mathrm{DMD}}[\Yv\Vv\mathbf{\Sigma}^{-1}\Wv] =\Yv \Vv \mathbf{\Sigma}^{-1}\underbrace{\Uv^*\Yv\Vv\mathbf{\Sigma}^{-1}}_{\widetilde{\Kv}_{\mathrm{DMD}}}\Wv= [\Yv\Vv\mathbf{\Sigma}^{-1}\Wv] \mathbf{\Lambda},
	$$
	and hence \cref{alg:DMD_vanilla} computes exact eigenvalues and eigenvectors of $\Kv_{\mathrm{DMD}}$. Moreover, one can show that this process identifies all of the nonzero eigenvalues of $\Kv_{\mathrm{DMD}}$ \cite[Thm. 1]{tu2014dynamic}. It is common to call $\Yv\Vv\mathbf{\Sigma}^{-1}\Wv$ \textit{exact modes} and $\Uv\Wv$ \textit{projected modes}.

\subsection{Linear dictionaries and the connection with EDMD}
\label{sec:galerkin_interp}

The connection between \cref{alg:DMD_vanilla} and Koopman operators is revealed once we interpret DMD as an instance of EDMD discussed in \cref{sec:EDMD}. We can think of the $j$th row of the POD matrix $\Uv^*\Xv$ as a linear function $\psi_j$ on the statespace $\Omega$ evaluated at the snapshot data:
$$
\psi_j(\xv)=[\Uv_{:,j}]^*\xv,\quad \psi_j(\xv^{(m)})=[\Uv^*\Xv]_{jm}.
$$
With this choice, and assuming that the SVD in \cref{alg:DMD_vanilla} is exact, the matrices in \cref{eq_EDMD_corr_matrices} become
$$
\Gv=\frac{1}{M} \Uv^*\Xv\Xv^*\Uv=\frac{1}{M}\Sigmav^2,\quad \Av = \frac{1}{M} \Uv^*\Xv\Yv^*\Uv=\frac{1}{M}\Sigmav\Vv^*\Yv^*\Uv.
$$
It follows that the EDMD matrix $\Kv$ satisfies
$$
\Kv=\Gv^{-1}\Av=\Sigmav^{-1}\Vv^*\Yv^*\Uv=\widetilde{\Kv}_{\mathrm{DMD}}^*.
$$
Hence, whilst $\Kv$ is an approximation of $\mathcal{K}$ in coefficient space, $\widetilde{\Kv}_{\mathrm{DMD}}$ is an approximation in statespace. This connection has been known since the original EDMD paper \cite{williams2015data}.

\subsection{Residuals}
\label{sec:exact_DMD_res}

Given the connection in \cref{sec:galerkin_interp}, a first attempt at combining ResDMD with exact DMD may be to consider the residual with respect to the new feature map $\Psiv(\xv)=\xv^*\Uv$. That is, for a vector $\vv\in \mathbb{C}^{M}$ and $\lambda\in\mathbb{C}$, consider the squared relative residual\footnote{Notice here that the contribution of the weights $w_j=1/M$ cancel since they are constant.}
\begin{equation}
\label{biagbfvuibi0}
\frac{\vv^*[\Uv^*\Yv\Yv^*\Uv-\lambda \Uv^*\Yv\Xv^*\Uv-\overline{\lambda}\Uv^*\Xv\Yv^*\Uv+|\lambda|^2\Uv^*\Xv\Xv^*\Uv]\vv}{\vv^*\Uv^*\Xv\Xv^*\Uv\vv^*}.
\end{equation}
The following proposition shows that this residual is not a good choice whenever $\Xv^*\Uv$ has full rank and $M\leq d$.

\begin{proposition}
\label{prop_first_res_rubbish}
Suppose that $M\leq d$, $\Xv^*\Uv$ has rank $M$ and that $\widetilde{\Kv}_{\mathrm{DMD}}$ is the output of \cref{alg:DMD_vanilla} with input $r=M$. Then for any eigenvalue-eigenvector pair $\lambda$ and $\vv\in \mathbb{C}^{M}$ of $\widetilde{\Kv}_{\mathrm{DMD}}^*$, the residual in \eqref{biagbfvuibi0} vanishes.
\end{proposition}

\begin{proof}
The numerator in \eqref{biagbfvuibi0} is the square of the $\ell^2$-vector norm of $(\Yv^*-\lambda \Xv^*)\Uv\vv$.
From the decomposition $\Xv = \Uv \mathbf{\Sigma}\Vv^*$, we see that
$$
(\Yv^*-\lambda \Xv^*)\Uv\vv
=\Yv^*\Uv\vv-\lambda \Vv\Sigmav\vv=\Vv\Sigmav(\Sigmav^{-1}\Vv^*\Yv^*\Uv-\lambda)\vv=\Vv\Sigmav(\widetilde{\Kv}_{\mathrm{DMD}}^*-\lambda)\vv=0.
$$
\end{proof}

The problem with the residual in \cref{biagbfvuibi0} is that in the regime $M\leq d$, a solution $\Kv_{\mathrm{DMD}}=\Yv\Xv^{\dagger}$ to the least squares problem in \cref{DMD_opt_vanilla} can be chosen so that
$$
\Yv-\Kv_{\mathrm{DMD}}\Xv=\Yv-\Yv\Xv^{\dagger}\Xv=\Yv-\Yv=0.
$$
This is because $\Xv^{\dagger}\Xv$ acts as the identity on $\mathbb{R}^{M}$ (assuming $\Xv$ has rank $M$).
To associate a residual with exact DMD, we express the $\widetilde{\Kv}_{\mathrm{DMD}}$ as a solution of a different least-squares problem. Dropping the the assumption that $r=M$ in \cref{alg:DMD_vanilla}, we still have $\Xv\Vv\Sigmav^{-1}=\Uv$ and hence
$$
\widetilde{\Kv}_{\mathrm{DMD}}=\Uv^*\Yv\Vv\Sigmav^{-1}=\Uv^\dagger\Yv\Vv\Sigmav^{-1}=(\Xv\Vv\Sigmav^{-1})^\dagger\Yv\Vv\Sigmav^{-1}.
$$
It follows that $\widetilde{\Kv}_{\mathrm{DMD}}$ is a solution to the least squares problem
\begin{equation}
\label{exact_DMD_opt_reg}
\min_{\Mv\in\mathbb{C}^{r\times r}} \left\|\Yv\Vv\Sigmav^{-1}-\Xv\Vv\Sigmav^{-1}\Mv\right\|_{\mathrm{F}}.
\end{equation}
This optimization problem no longer has the vanishing residuals problem. Hence, given any eigenpair $(\lambda,\vv)\in\mathbb{C}\times \mathbb{C}^r$ of $\widetilde{\Kv}_{\mathrm{DMD}}$, we consider the relative residual
$$
{\left\|\Yv\Vv\Sigmav^{-1}\vv-{\lambda}\Xv\Vv\Sigmav^{-1}\vv\right\|_{\ell^2}}\big/{\left\|\vv\right\|_{\ell^2}}.
$$
Here, we have used the fact that $\left\|\Xv\Vv\Sigmav^{-1}\vv\right\|_{\ell^2}=\|\Uv\vv\|_{\ell^2}=\|\vv\|_{\ell^2}$ in the normalization. We may evaluate this residual directly, or we can use the fact that
$$
\left\|\Yv\Vv\Sigmav^{-1}\vv-{\lambda}\Xv\Vv\Sigmav^{-1}\vv\right\|_{\ell^2}^2=(\Vv\Sigmav^{-1}\vv)^*\left[\Yv^*\Yv-\lambda \Yv^*\Xv-\overline{\lambda}\Xv^*\Yv+|\lambda|^2\Xv^*\Xv\right](\Vv\Sigmav^{-1}\vv).
$$
This expression simplifies since
$$
(\Vv\Sigmav^{-1})^*\Xv^*\Yv(\Vv\Sigmav^{-1})=\widetilde{\Kv}_{\mathrm{DMD}},\quad (\Vv\Sigmav^{-1})^*\Xv^*\Xv(\Vv\Sigmav^{-1})=\Uv^*\Uv=\Iv_{r},
$$
where $\Iv_r$ denotes the $r\times r$ identity matrix. Letting $\widetilde{\Lv}_{\mathrm{DMD}}=(\Yv\Vv\Sigmav^{-1})^*(\Yv\Vv\Sigmav^{-1})$, we arrive at the relative residual
$$
\mathrm{res}(\lambda,\vv)=
{\sqrt{\vv^*[\widetilde{\Lv}_{\mathrm{DMD}}-\lambda\widetilde{\Kv}_{\mathrm{DMD}}^*-\overline{\lambda}\widetilde{\Kv}_{\mathrm{DMD}}+|\lambda|^2\Iv_r]\vv}}\Big/{\|\vv\|_{\ell^2}}
$$
for a candidate eigenpair $(\lambda,\vv)\in\mathbb{C}\times \mathbb{C}^r$ of $\widetilde{\Kv}_{\mathrm{DMD}}$.

\begin{algorithm}[t]
\textbf{Input:}  Snapshot data $\Xv\in\mathbb{R}^{d\times M}$ and $\Yv\in\mathbb{R}^{d\times M}$, rank $r\in\mathbb{N}$. \\
\vspace{-4mm}
\begin{algorithmic}[1]
\State Run exact DMD (\cref{alg:DMD_vanilla}).
\State Compute the matrix $\widetilde{\Lv}_{\mathrm{DMD}}=(\Yv\Vv\Sigmav^{-1})^*(\Yv\Vv\Sigmav^{-1})$.
\State For each eigenpair $(\lambda_j,\wv_{j})$, compute the residual
$$
r_j=\mathrm{res}(\lambda_j,\wv_{j})=\sqrt{\frac{\wv_j^*\widetilde{\Lv}_{\mathrm{DMD}}\wv_j}{\|\wv_j\|_{\ell^2}^2}-|\lambda_j|^2}.
$$
\end{algorithmic} \textbf{Output:} Matrix $\widetilde{\Kv}_{\mathrm{DMD}}$, the eigenvalues $\mathbf{\Lambda}$ and eigenvector coefficients $\Wv$ and residuals $\{r_j\}_{j=1}^r$.
\caption{ResDMD algorithm for validating eigenpairs of exact DMD. Given a threshold $\epsilon>0$, we recommend discarding eigen-triples $(\lambda_j,\wv_{j})$ with $r_j>\epsilon$.}
\label{alg:exact_ResDMD1}
\end{algorithm}

This residual naturally leads to two different flavors of algorithms:
\begin{itemize}
	\item We can validate eigenpairs of exact DMD and disregard spurious eigenpairs. The process is summarized in \cref{alg:exact_ResDMD1}. Notice that if $\widetilde{\Kv}_{\mathrm{DMD}}\wv=\lambda\wv$ then the residual simplifies to
	$$
	\mathrm{res}(\lambda,\wv)=
	\sqrt{\frac{\wv^*\widetilde{\Lv}_{\mathrm{DMD}}\wv}{\|\wv\|_{\ell^2}^2}-|\lambda|^2}
	$$
	Given a threshold $\epsilon>0$, we recommend discarding eigenpairs with a residual larger than $\epsilon$.
	\item We can compute pseudospectra and for any $\lambda\in\mathbb{C}$, compute an approximate eigenfunction by selecting $\vv$ to minimize $\mathrm{res}(\lambda,\vv)$. This process is summarized in \cref{alg:exact_ResDMD2,alg:exact_ResDMD3}.
\end{itemize}

\begin{algorithm}[t]
\textbf{Input:}  Snapshot data $\Xv\in\mathbb{R}^{d\times M}$, $\Yv\in\mathbb{R}^{d\times M}$, rank $r\in\mathbb{N}$, accuracy goal $\epsilon>0$, grid $z_1,\ldots,z_k\in\mathbb{C}$.\\
\vspace{-4mm}
\begin{algorithmic}[1]
\State Run exact DMD (\cref{alg:DMD_vanilla}).
\State Compute the matrix $\Lv_{\mathrm{DMD}}=(\Yv\Vv\Sigmav^{-1})^*(\Yv\Vv\Sigmav^{-1})$.
\State For each $z_j$, compute
$$
\tau_j=\min_{\vv\in\mathbb{C}^r} {\sqrt{\vv^*[\widetilde{\Lv}_{\mathrm{DMD}}-z_j\widetilde{\Kv}_{\mathrm{DMD}}^*-\overline{z_j}\widetilde{\Kv}_{\mathrm{DMD}}+|z_j|^2\Iv_r]\vv}}\Big/{\|\vv\|_{\ell^2}},
$$
which is an SVD problem, and the corresponding (right) singular vectors $\vv_j$.
\end{algorithmic} \textbf{Output:} Approximate $\epsilon$-pseudospectrum $\{z_j: \tau_j<\epsilon\}$, approximate eigenfunctions $\{\vv_j: \tau_j<\epsilon\}$.\caption{ResDMD algorithm using exact DMD for computation of pseudospectra.}
\label{alg:exact_ResDMD2}
\end{algorithm}

\begin{algorithm}[t]
\textbf{Input:}  Snapshot data $\Xv\in\mathbb{R}^{d\times M}$, $\Yv\in\mathbb{R}^{d\times M}$, rank $r\in\mathbb{N}$, $\lambda\in\mathbb{C}$.\\
\vspace{-4mm}
\begin{algorithmic}[1]
\State Run exact DMD (\cref{alg:DMD_vanilla}).
\State Compute the matrix $\Lv_{\mathrm{DMD}}=(\Yv\Vv\Sigmav^{-1})^*(\Yv\Vv\Sigmav^{-1})$.
\State Compute
$$
r=\min_{\vv\in\mathbb{C}^r} {\sqrt{\vv^*[\widetilde{\Lv}_{\mathrm{DMD}}-\lambda\widetilde{\Kv}_{\mathrm{DMD}}^*-\overline{\lambda}\widetilde{\Kv}_{\mathrm{DMD}}+|\lambda|^2\Iv_r]\vv}}\Big/{\|\vv\|_{\ell^2}},
$$
which is an SVD problem, and the corresponding (right) singular vector $\vv$.
\end{algorithmic} \textbf{Output:} 
Residual $r$ and approximate eigenfunction $\vv$.
\caption{ResDMD algorithm using exact DMD for computation of approximate eigenfunction.}
\label{alg:exact_ResDMD3}
\end{algorithm}

\section{ResDMD for kernelized DMD methods}
\label{sec:kernelized_ResDMD}

We now generalize the discussion in \cref{sec:exact_ResDMD} to kernelized EDMD. This allows us to efficiently deal with nonlinear observables in our dictionary, even when the statespace dimension $d$ is large.

\subsection{Kernelized EDMD}

A naïve construction of the matrix $\Psiv_X^*\Wv\Psiv_Y$ in \cref{sec:EDMD} requires $\mathcal{O}(N^2M)$ operations, which becomes impractical when $N$ is large. As an example, the paper~\cite{williams2015kernel} points out that if $V_N$ is the space of all multivariate polynomials on\footnote{For many modern applications, this is a modest ambient statespace dimension!} $\mathbb{R}^{256}$ with degree up to $20$, then $N\sim 10^{30}$. This number is far too large even to store the EDMD matrix, let alone compute it. This is an example of the curse of dimensionality. Kernelized EDMD~\cite{williams2015kernel} aims to make EDMD practical for large (even infinite) $N$. The idea is to compute a much smaller matrix $\widehat{\Kv}$ with the same nonzero eigenvalues as ${\Kv}$. The fundamental proposition is the following.

\begin{proposition}[Proposition 1 of~\cite{williams2015kernel}]
\label{prop_kern_EDMD}
Let $\sqrt{\Wv}\Psiv_X = \Qv\Sigmav \Zv^*$ be an SVD, where $\Qv\in\mathbb{C}^{M\times M}$ is a unitary matrix, $\Sigmav\in\mathbb{R}^{M\times M}$ is a diagonal matrix with nonincreasing nonnegative entries, and $\Zv\in \mathbb{C}^{N\times M}$ is an isometry. Define the matrix
$$
\widehat{\Kv}=(\Sigmav^\dagger \Qv^*)\left(\sqrt{\Wv}\Psiv_Y\Psiv_X^*\sqrt{\Wv}\right)(\Qv\Sigmav^\dagger)\in\mathbb{C}^{M\times M}.
$$
Then, for $\lambda\neq 0$ and $\vv\in \mathbb{C}^{M}$, $(\lambda,\vv)$ is an eigenvalue-eigenvector pair of $\widehat{\Kv}$ if and only if $(\lambda,\Zv\vv)$ is an eigenvalue-eigenvector pair of $\Kv$.
\end{proposition}

To make use of this proposition, let 
$$
\widehat{\Gv}=\sqrt{\Wv}\Psiv_X\Psiv_X^*\sqrt{\Wv}\in\mathbb{C}^{M\times M},\qquad \widehat{\Av}=\sqrt{\Wv}\Psiv_Y\Psiv_X^*\sqrt{\Wv}\in\mathbb{C}^{M\times M}.
$$
The matrices $\Qv$ and $\Sigmav$ in \cref{prop_kern_EDMD} can be recovered from the eigenvalue decomposition $\widehat{\Gv}=\Qv\Sigmav^2\Qv^*$. 
Moreover, both matrices $\widehat{\Gv}$ and $\widehat{\Av}$ can be computed using inner products:
	\begin{equation}
	\label{kern_trick1}
\widehat{\Gv}_{jk}=\sqrt{w_j}\Psiv(\xv^{(j)})\Psiv(\xv^{(k)})^*\sqrt{w_k},\qquad
\widehat{\Av}_{jk}=\sqrt{w_j}\Psiv(\yv^{(j)})\Psiv(\xv^{(k)})^*\sqrt{w_k},
	\end{equation}
where we recall that $\Psiv(\xv)$ is a row vector of the dictionary evaluated at $\xv$.\footnote{This is the transpose of the convention in~\cite{williams2015kernel}.} Kernelized EDMD applies the kernel trick to compute the inner products in \eqref{kern_trick1} in an implicitly defined reproducing Hilbert space $\mathcal{H}$ with inner product $\langle\cdot,\cdot\rangle_{\mathcal{H}}$~\cite{scholkopf2001kernel}. A positive-definite kernel function $\mathcal{S}:\Omega\times \Omega\rightarrow\mathbb{R}$ induces a feature map $\varphi:\mathbb{R}^d\rightarrow\mathcal{H}$ so that
$$
\langle\varphi(\xv),\varphi(\xv')\rangle_{\mathcal{H}}=\mathcal{S}(\xv,\xv').
$$
Often $\mathcal{S}$ can be evaluated in $\mathcal{O}(d)$ operations.
This leads to a choice of dictionary (equivalently, a reweighted feature map) so that the matrices
$$
\widehat{\Gv}_{jk}=\mathcal{S}(\xv^{(j)},\xv^{(k)}),\qquad \widehat{\Av}_{jk}=\mathcal{S}(\yv^{(j)},\xv^{(k)})
$$
can be computed in $\mathcal{O}(dM^2)$ operations. $\widehat{\Kv}$ can thus be constructed in $\mathcal{O}(dM^2)$ operations, a considerable saving, with a significant reduction in memory consumption.

Different choices of kernel result in different matrices $\widehat{\Kv}$, and in general different eigenvectors and eigenfunctions. Common choices of the kernel $\mathcal{S}$ include:
\begin{itemize}
	\item Polynomial kernel: $\mathcal{S}(\xv,\xv')=(\xv'^*\xv/c^2+1)^\alpha$ for a scaling factor $c$ and power $\alpha$;
	\item Gaussian kernel: $\mathcal{S}(\xv,\xv')=\exp(-\|\xv-\xv'\|^2/c^2)$;
	\item Laplacian kernel: $\mathcal{S}(\xv,\xv')=\exp(-\|\xv-\xv'\|/c)$;
	\item Lorentzian kernel: $\mathcal{S}(\xv,\xv')=(1+\|\xv-\xv'\|^2/c^2)^{-1}$.
\end{itemize}
One way of viewing kernelized EDMD is that it is EDMD with the new feature map
$
\xv\mapsto \Psiv(\xv)\Zv.
$
We might not want to include the whole SVD of $\sqrt{\Wv}\Psiv_X$ and instead keep only $r$ leading principal components. This often helps avoid spurious unstable eigenvalues that DMD and EDMD can produce. The process is summarized in \cref{alg:kEDMD_vanilla}.

\begin{algorithm}[t]
\textbf{Input:} Snapshot data $\Xv\in\mathbb{R}^{d\times M}$ and $\Yv\in\mathbb{R}^{d\times M}$, kernel $\mathcal{S}$, rank $r\in\mathbb{N}$. \\
\vspace{-4mm}
\begin{algorithmic}[1]
\State Compute the matrices $\widehat{\Gv}\in\mathbb{R}^{M\times M}$ and $\widehat{\Av}\in\mathbb{R}^{M\times M}$ via
$
\widehat{\Gv}_{jk}=\mathcal{S}(\xv^{(j)},\xv^{(k)}), \widehat{\Av}_{jk}=\mathcal{S}(\yv^{(j)},\xv^{(k)}).
$
\State Compute the eigendecomposition $\widehat{\Gv}=\Qv\Sigmav^2\Qv^*$.
\State Let $\widehat{\Sigmav}=\Sigmav(1:r,1:r)\in\mathbb{R}^{r\times r}$ and $\widehat{\Qv}=\Qv(:,1:r)\in\mathbb{R}^{M\times r}$ be the matrices of the $r$ dominant eigenvalues and eigenvectors, respectively.
\State Compute the compression $\widehat{\Kv}=(\widehat{\Sigmav}^\dagger \widehat{\Qv}^*)\widehat{\Av}(\widehat{\Qv}\widehat{\Sigmav}^\dagger)\in\mathbb{R}^{r\times r}$.
\end{algorithmic} \textbf{Output:} Matrix $\widehat{\Kv}$.
\caption{The kernelized EDMD algorithm \cite{williams2015kernel}.}
\label{alg:kEDMD_vanilla}
\end{algorithm}

\subsection{Kernelized ResDMD: Galerkin perspective}

Our question now shifts to how to use the kernel trick with ResDMD. A first attempt may be to consider the residual with respect to the new feature map $\Psiv\Zv$. That is, for a vector $\vv\in \mathbb{C}^{M}$ and $\lambda\in\mathbb{C}$, consider the squared relative residual
\begin{equation}
\label{biagbfvuibi}
\frac{\vv^*[(\Psiv_Y\Zv)^*\Wv(\Psiv_Y\Zv)\!-\!\lambda (\Psiv_Y\Zv)^*\Wv(\Psiv_X\Zv)\!-\!\overline{\lambda}(\Psiv_X\Zv)^*\Wv(\Psiv_Y\Zv)\!+\!|\lambda|^2(\Psiv_X\Zv)^*\Wv(\Psiv_X\Zv)]\vv}{\vv^*(\Psiv_X\Zv)^*\Wv(\Psiv_X\Zv)\vv^*}.
\end{equation}
However, this is not a good choice in the regime $M\leq N$. The following proposition from \cite{colbrook2021rigorous} is analogous to \cref{prop_first_res_rubbish}.

\begin{proposition}
Suppose that $M\leq N$ and $\sqrt{\Wv}\Psiv_X$ has rank $M$ (independent rows). For any eigenvalue-eigenvector pair $\lambda$ and $\vv\in \mathbb{C}^{M}$ of $\widehat{\Kv}$, the residual in \eqref{biagbfvuibi} vanishes.
\end{proposition}

In other words, the restriction $M\leq N$ prevents the large data convergence $(M\rightarrow\infty)$ for a fixed $N$. One way to overcome this issue is to split the snapshot data into two sets: one that is used to compute a suitable dictionary and another to evaluate the dictionary at data points for quadrature. In particular, we can evaluate the new dictionary at any point $\xv\in\Omega$ via
$$
\sqrt{\Wv}\Psiv(\xv)\Zv=\sqrt{\Wv}\Psiv(\xv)(\sqrt{\Wv}\Psiv_X)^*\Qv\Sigmav^{\dagger}=\begin{bmatrix}
\mathcal{S}(\xv,\xv^{(1)})&\mathcal{S}(\xv,\xv^{(2)})&\cdots & \mathcal{S}(\xv,\xv^{(M)})
\end{bmatrix}\Qv\Sigmav^{\dagger}.
$$
The method of using two subsets of snapshots is explored extensively in \cite{colbrook2021rigorous,colbrook2023residualJFM,colbrook2023beyond} (including for exact DMD) and is not the focus of this paper.

\subsection{Kernelized ResDMD: Regression perspective}

Following \cref{sec:exact_DMD_res}, there is another way to associate a residual with a candidate eigenpair of $\widehat{\Kv}$. Suppose we take general $r\leq M$ in \cref{alg:kEDMD_vanilla}. We first define
$$
\widehat{\Psiv}_X=(\sqrt{\Wv}\Psiv_X)^*\widehat{\Qv}\widehat{\Sigmav}^{\dagger}=\Zv(:,1:r),\qquad \widehat{\Psiv}_Y=(\sqrt{\Wv}\Psiv_Y)^*\widehat{\Qv}\widehat{\Sigmav}^{\dagger}.
$$
Let $\widehat{\Zv}=\Zv(:,1:r)$, then we note that
$$
\widehat{\Psiv}_X^\dagger\widehat{\Psiv}_Y=\widehat{\Zv}^*(\sqrt{\Wv}\Psiv_Y)^*\widehat{\Qv}\widehat{\Sigmav}^{\dagger}=(\widehat{\Sigmav}^\dagger \widehat{\Qv}^*)\left(\sqrt{\Wv}\Psiv_X\Psiv_Y^*\sqrt{\Wv}\right)(\widehat{\Qv}\widehat{\Sigmav}^{\dagger})=\widehat{\Kv}^*.
$$
It follows that $\widehat{\Kv}^*$ is a solution to the least squares problem
\begin{equation}
\label{kEDMD_opt_reg}
\min_{\Mv\in\mathbb{C}^{r\times r}} \left\|\widehat{\Psiv}_Y-\widehat{\Psiv}_X\Mv\right\|_{\mathrm{F}}.
\end{equation}
Suppose that $(\lambda,\vv)\in\mathbb{C}\times \mathbb{C}^r$ is a \textit{left} eigenvector pair of $\widehat{\Kv}$, i.e.,
$
\vv^*\widehat{\Kv}=\lambda \vv^*.
$
Then
$
\widehat{\Kv}^*\vv=\overline{\lambda} \vv.
$
Hence, we consider the relative residual
\begin{align*}
{\left\|\widehat{\Psiv}_Y\vv-\overline{\lambda}\widehat{\Psiv}_X\vv\right\|_{\ell^2}}\big/{\left\|\vv\right\|_{\ell^2}}&=\sqrt{(\widehat{\Qv}\widehat{\Sigmav}^{\dagger}\vv)^*\left[\sqrt{\Wv}\Psiv_Y\Psiv_Y^*\sqrt{\Wv}-\lambda \widehat{\Av}^*-\overline{\lambda}\widehat{\Av}+|\lambda|^2\widehat{\Gv}\right](\widehat{\Qv}\widehat{\Sigmav}^{\dagger}\vv)}\Big/{\left\|\vv\right\|_{\ell^2}}\\
&=\sqrt{\vv^*\widehat{\Lv}\vv
-\lambda \vv^*\widehat{\Kv}^*\vv
-\overline{\lambda}\vv^*\widehat{\Kv}\vv
+|\lambda|^2\vv^*\vv}\Big/{\left\|\vv\right\|_{\ell^2}}
\end{align*}
Here, we have used $\|\widehat{\Psiv}_X\vv\|_{\ell^2}=\|\widehat{Z}\vv\|_{\ell^2}=\|\vv\|_{\ell^2}$ in the normalization, defined the matrix
$$
\widehat{\Lv}=(\widehat{\Qv}\widehat{\Sigmav}^{\dagger})^*\sqrt{\Wv}\Psiv_Y\Psiv_Y^*\sqrt{\Wv}(\widehat{\Qv}\widehat{\Sigmav}^{\dagger})\in\mathbb{C}^{r\times r}
$$
and used the fact that
$$
\widehat{\Psiv}_X^*\widehat{\Psiv}_X=(\widehat{\Qv}\widehat{\Sigmav}^{\dagger})^*\widehat{\Gv}(\widehat{\Qv}\widehat{\Sigmav}^{\dagger})=\Iv_r,\quad
\widehat{\Psiv}_Y^*\widehat{\Psiv}_X=(\widehat{\Qv}\widehat{\Sigmav}^{\dagger})^*\widehat{\Av}(\widehat{\Qv}\widehat{\Sigmav}^{\dagger})=\widehat{\Kv}.
$$
Notice that, as well as $\widehat{\Kv}$, the matrix $\widehat{\Lv}$ can also be kernelized since
$$
[\sqrt{\Wv}\Psiv_Y\Psiv_Y^*\sqrt{\Wv}]_{jk}=\mathcal{S}(\yv^{(j)},\yv^{(k)}).
$$
\begin{algorithm}[t]
\textbf{Input:} Snapshot data $\Xv\in\mathbb{R}^{d\times M}$ and $\Yv\in\mathbb{R}^{d\times M}$, kernel $\mathcal{S}$, rank $r\in\mathbb{N}$. \\
\vspace{-4mm}
\begin{algorithmic}[1]
\State Run kernelized EDMD (\cref{alg:kEDMD_vanilla}).
\State Compute the matrix $\widehat{\Lv}=(\widehat{\Qv}\widehat{\Sigmav}^{\dagger})^*\Mv(\widehat{\Qv}\widehat{\Sigmav}^{\dagger})$, where $\Mv_{jk}=\mathcal{S}(\yv^{(j)},\yv^{(k)})$.
\State Compute the eigendecomposition
$$
\widehat{\Kv}\Vv_{R}=\Vv_{R}\mathbf{\Lambda},\quad
\Vv_{L}^*\widehat{\Kv}=\mathbf{\Lambda}\Vv_{L}^*.
$$
The columns of $\Vv_{L}=[\vv_{1,L}\cdots \vv_{r,L}]$ and $\Vv_{R}=[\vv_{1,R}\cdots \vv_{r,R}]$ are (left and right) eigenvector coefficients and $\mathbf{\Lambda}$ is a diagonal matrix of eigenvalues.
\State For each eigen-triple $(\lambda_j,\vv_{j,L},\vv_{j,R})$, compute the residual
$$
r_j=\mathrm{res}(\lambda_j,\vv_{j,L},\vv_{j,R})=\sqrt{\frac{\vv_{j,L}^*\left[\widehat{\Lv}-\lambda \widehat{\Kv}^*-\overline{\lambda}\widehat{\Kv}+|\lambda|^2\Iv_r\right]\vv_{j,L}}{\vv_{j,L}^*\vv_{j,L}}}=\sqrt{\frac{\vv_{j,L}^*\widehat{\Lv}\vv_{j,L}}{\|\vv_{j,L}\|_{\ell^2}^2}-|\lambda|^2}
$$
\end{algorithmic} \textbf{Output:} Matrix $\widehat{\Kv}$, the eigenvalues $\mathbf{\Lambda}$ and eigenvector coefficients $\Vv_{L},\Vv_{R}$ and residuals $\{r_j\}_{j=1}^r$.\caption{The kernelized ResDMD algorithm for validating eigenpairs. Given a threshold $\epsilon>0$, we recommend discarding eigen-triples $(\lambda_j,\vv_{j,L},\vv_{j,R})$ with $r_j>\epsilon$.}
\label{alg:kernel_ResDMD1}
\end{algorithm}

\begin{algorithm}[t]
\textbf{Input:}  Snapshot data $\Xv\in\mathbb{R}^{d\times M}$, $\Yv\in\mathbb{R}^{d\times M}$, kernel $\mathcal{S}$, rank $r\in\mathbb{N}$, accuracy goal $\epsilon>0$, grid $z_1,\ldots,z_k\in\mathbb{C}$.\\
\vspace{-4mm}
\begin{algorithmic}[1]
\State Run kernelized EDMD (\cref{alg:kEDMD_vanilla}).
\State Compute the matrix $\widehat{\Lv}=(\widehat{\Qv}\widehat{\Sigmav}^{\dagger})^*\Mv(\widehat{\Qv}\widehat{\Sigmav}^{\dagger})$, where $\Mv_{jk}=\mathcal{S}(\yv^{(j)},\yv^{(k)})$.
\State For each $z_j$, compute
$$
\tau_j=\min_{\vv\in\mathbb{C}^r} {\sqrt{\vv^*[\widehat{\Lv}-z_j\widehat{\Kv}^*-\overline{z_j}\widehat{\Kv}+|z_j|^2\Iv_r]\vv}}\Big/{\|\vv\|_{\ell^2}},
$$
which is an SVD problem, and the corresponding (right) singular vectors $\vv_j$.
\end{algorithmic} \textbf{Output:} Approximate $\epsilon$-pseudospectrum $\{z_j: \tau_j<\epsilon\}$, approximate eigenfunctions $\{\vv_j: \tau_j<\epsilon\}$.\caption{ResDMD algorithm using kernelized EDMD for computation of pseudospectra.}
\label{alg:kernel_ResDMD2}
\end{algorithm}

\begin{algorithm}[t]
\textbf{Input:}  Snapshot data $\Xv\in\mathbb{R}^{d\times M}$, $\Yv\in\mathbb{R}^{d\times M}$, kernel $\mathcal{S}$, rank $r\in\mathbb{N}$, $\lambda\in\mathbb{C}$.\\
\vspace{-4mm}
\begin{algorithmic}[1]
\State Run kernelized EDMD (\cref{alg:kEDMD_vanilla}).
\State Compute the matrix $\widehat{\Lv}=(\widehat{\Qv}\widehat{\Sigmav}^{\dagger})^*\Mv(\widehat{\Qv}\widehat{\Sigmav}^{\dagger})$, where $\Mv_{jk}=\mathcal{S}(\yv^{(j)},\yv^{(k)})$.
\State Compute
$$
r=\min_{\vv\in\mathbb{C}^r} {\sqrt{\vv^*[\widehat{\Lv}-\lambda\widehat{\Kv}^*-\overline{\lambda}\widehat{\Kv}+|\lambda|^2\Iv_r]\vv}}\Big/{\|\vv\|_{\ell^2}},
$$
which is an SVD problem, and the corresponding (right) singular vector $\vv$.
\end{algorithmic} \textbf{Output:} 
Residual $r$ and approximate eigenfunction $\vv$.\caption{ResDMD algorithm using kernelized EDMD for computation of approximate eigenfunction.}
\label{alg:kernel_ResDMD3}
\end{algorithm}

As in the case of exact DMD, this residual naturally leads to two different flavors of algorithms:
\begin{itemize}
	\item We can validate eigenpairs of kernelized EDMD and disregard spurious eigenpairs. The process is summarized in \cref{alg:kernel_ResDMD1}. If $\vv^*\widehat{\Kv}=\lambda\vv^*$, then the residual simplifies to
	$$
	\sqrt{\frac{\vv^*\widehat{\Lv}\vv}{\|\vv\|_{\ell^2}^2}-|\lambda|^2}
	$$
We recommend discarding eigen-triples with a residual larger than a threshold $\epsilon>0$.
	\item We can compute pseudospectra and for any $\lambda\in\mathbb{C}$, compute an approximate eigenfunction by selecting $\vv$ to minimize the residual. This process is summarized in \cref{alg:kernel_ResDMD2,alg:kernel_ResDMD3}.
\end{itemize}

\section{Three examples}
\label{sec:examples}

We now present three examples of the new residuals for ResDMD. The first and last use kernelized EDMD, and the second uses exact DMD.

\subsection{Example 1: Cylinder wake}

We first consider the classic example of a low Reynolds number flow past a circular cylinder. This is one of the most studied examples in modal-analysis techniques \cite[Table 3]{rowley2017model}\cite{chen2012variants,taira2020modal}. $Re=100$ is chosen so that it is larger than the critical Reynolds number at which the flow undergoes a supercritical Hopf bifurcation, resulting in laminar vortex shedding \cite{jackson1987finite,zebib1987stability}. This limit cycle is stable and is representative of the three-dimensional flow \cite{noack1994global,noack2003hierarchy}. The Koopman operator of the post-transient flow has a pure point spectrum with a lattice structure on the unit circle \cite{bagheri2013koopman}.

We numerically compute the velocity field of a flow around a circular cylinder of diameter $D=1$ using an incompressible, two-dimensional lattice-Boltzmann solver \cite{jozsa2016validation,szHoke2017performance}. The temporal resolution of the flow is chosen so that approximately 24 snapshots of the flow field correspond to the period of vortex shedding. The computational domain size is $18D$ in length and $5D$ in height, with a $800\times 200$ grid resolution. The cylinder is positioned $2D$ downstream of the inlet at the mid-height of the domain. The cylinder side walls are defined as bounce-back and no-slip walls, and a parabolic velocity profile is given at the inlet of the domain. The outlet is defined as a non-reflecting outflow. After simulations converge to steady-state vortex shedding, we collect $M$ snapshots of the vorticity field. We then use kernelized EDMD with the Laplacian kernel $\mathcal{S}(\xv,\xv')=\exp(-\|\xv-\xv'\|/c)$, where $c$ is the average $\ell_2$-norm of the mean-subtracted snapshot.

\cref{fig:cylinder1} shows the eigenvalues and residuals (computed using \cref{alg:kernel_ResDMD1}) for several values of $M$. For this example, the system is measure-preserving, and hence, we expect that $\vv_{j,L}^*\widehat{\Lv}\vv_{j,L}$ is close to $\vv_{j,L}^*\vv_{j,L}$. Hence, we expect the residuals to be close to $\sqrt{1-|\lambda|^2}$. This is verified in the bottom row of the figure. \cref{fig:cylinder2} shows the pseudospectrum computed using \cref{alg:kernel_ResDMD2}. We see convergence to the pseudospectra of the Koopman operator. In contrast, if we use the naive residual in \cref{biagbfvuibi}, the approximations of pseudospectra are completely wrong and do not allow us to detect spurious eigenvalues. This is shown in \cref{fig:cylinder3}.

\begin{figure}
\centering
\includegraphics[width=0.24\textwidth]{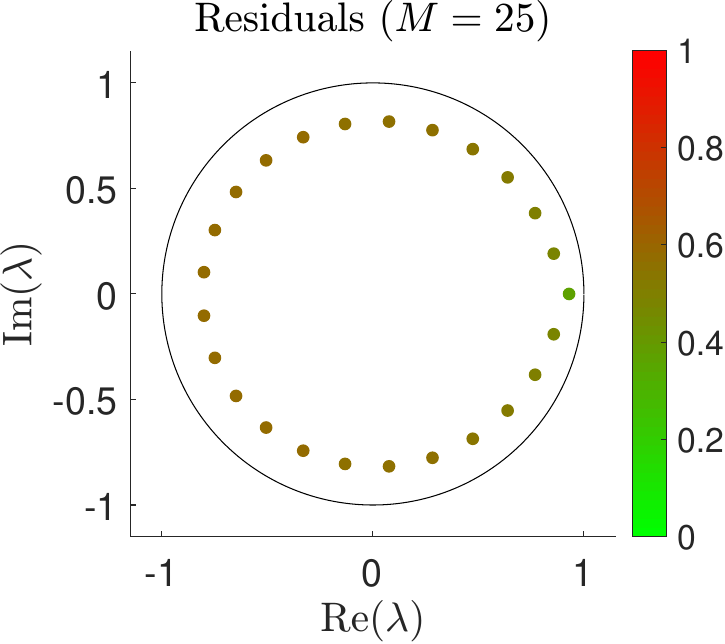}\hfill
\includegraphics[width=0.24\textwidth]{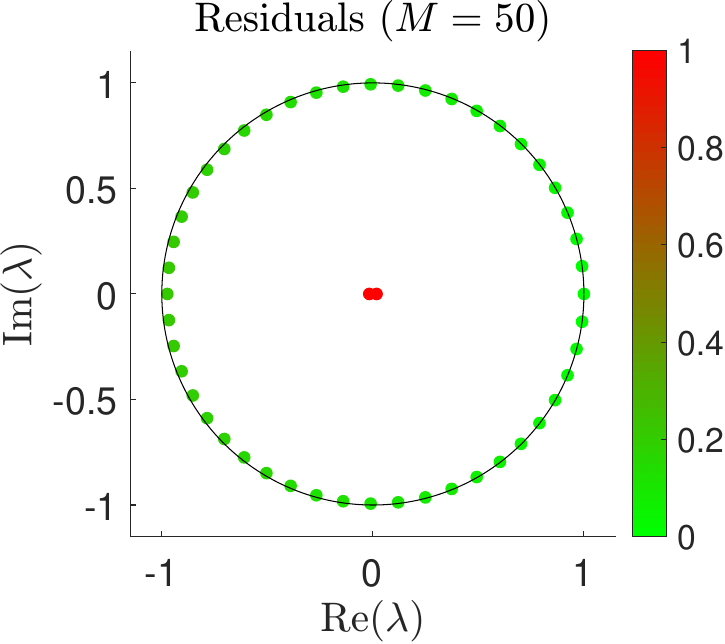}\hfill
\includegraphics[width=0.24\textwidth]{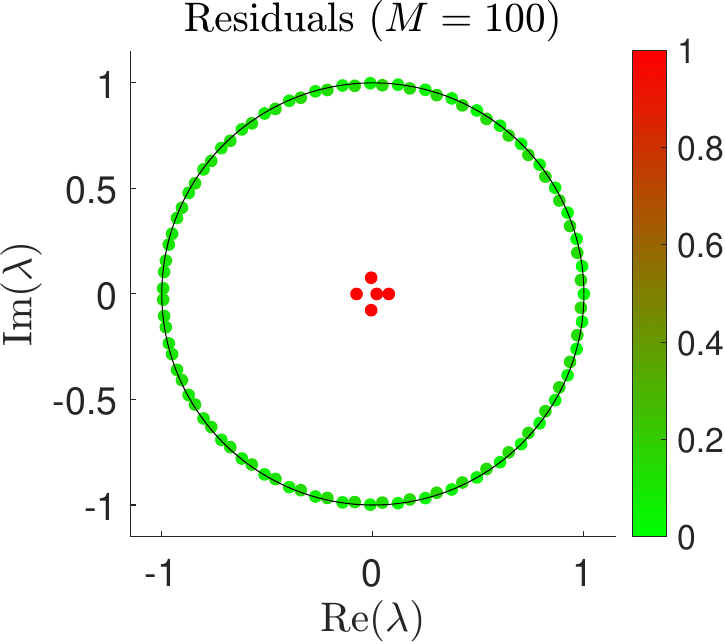}\hfill
\includegraphics[width=0.24\textwidth]{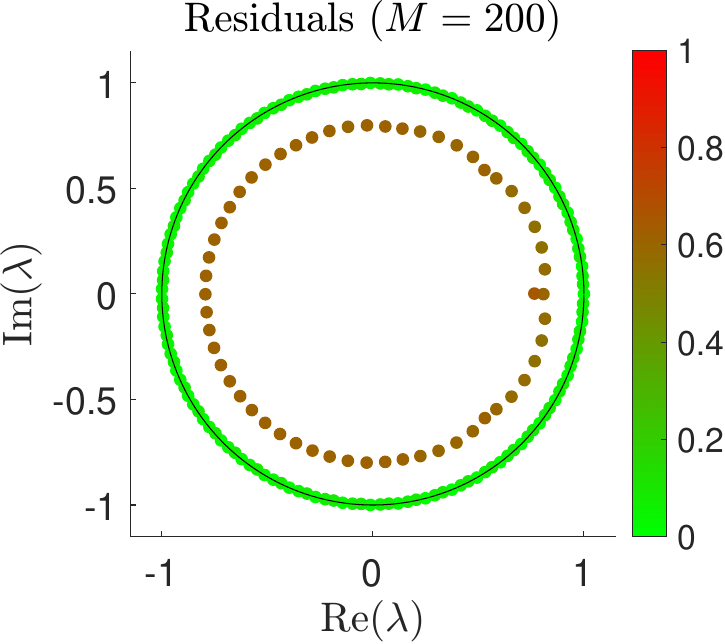}\\
\includegraphics[width=0.24\textwidth]{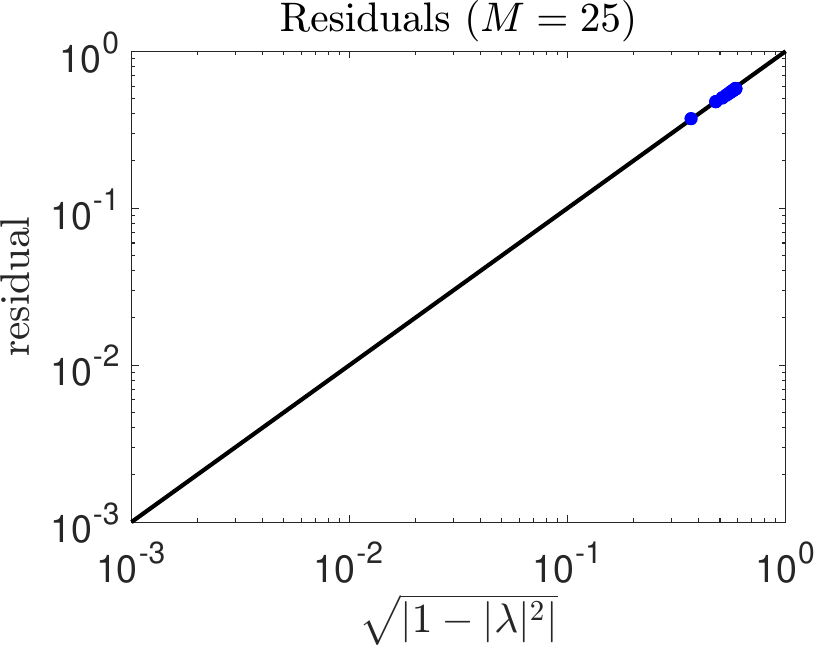}\hfill
\includegraphics[width=0.24\textwidth]{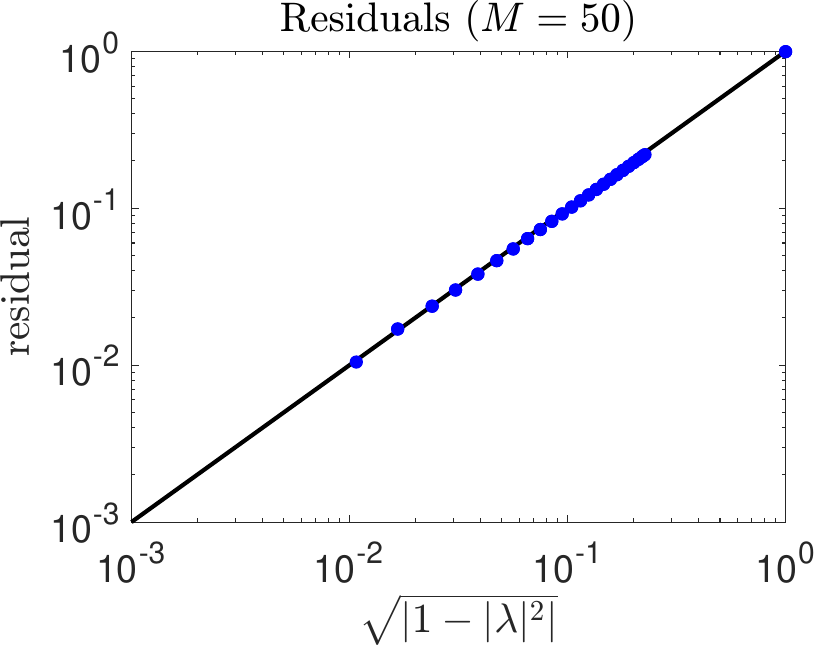}\hfill
\includegraphics[width=0.24\textwidth]{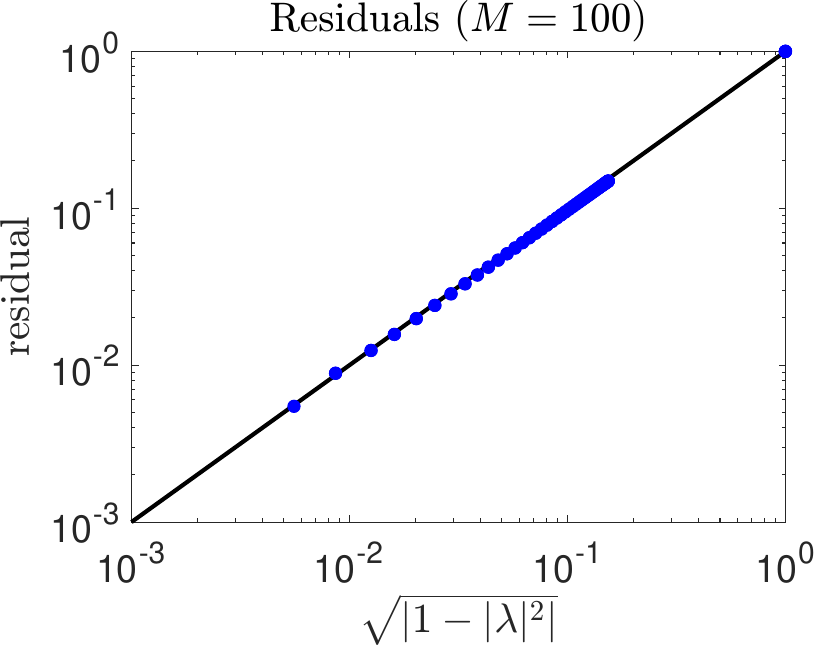}\hfill
\includegraphics[width=0.24\textwidth]{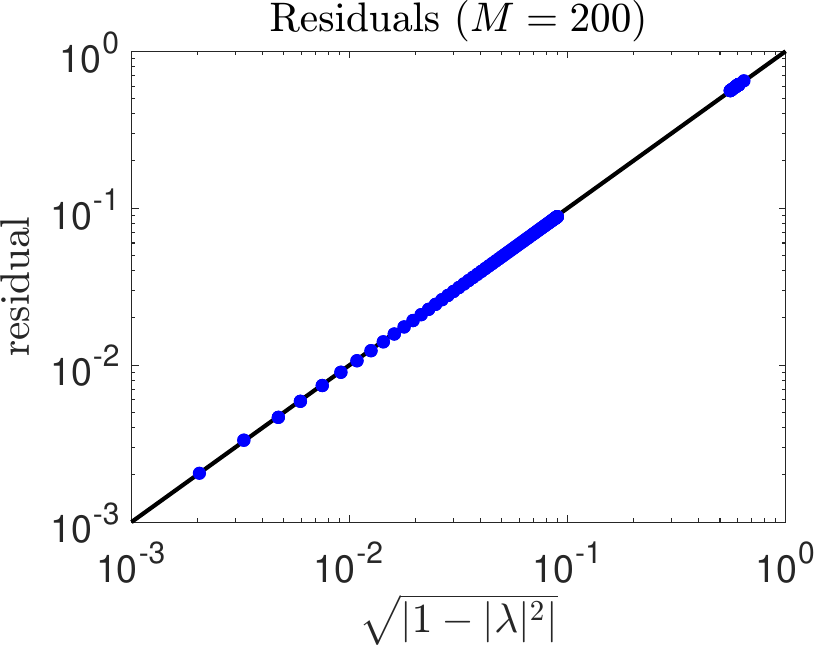}
\caption{Residuals computed using \cref{alg:kernel_ResDMD1} for the cylinder flow. The top row shows the location of the eigenvalues and the residuals. The bottom row plots the residuals against $\sqrt{1-|\lambda|^2}$ with the black line corresponding to an exact relation.}
\label{fig:cylinder1}
\end{figure}

\begin{figure}
\centering
\includegraphics[width=0.24\textwidth]{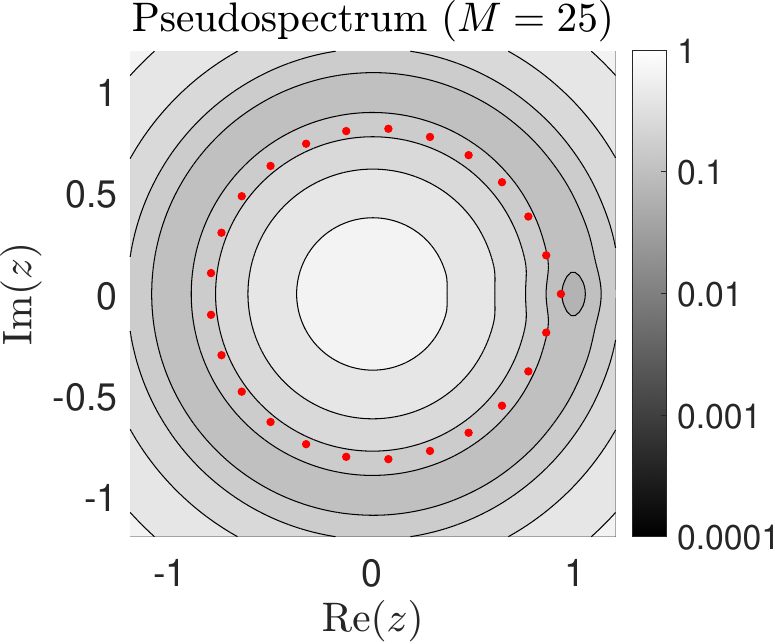}\hfill
\includegraphics[width=0.24\textwidth]{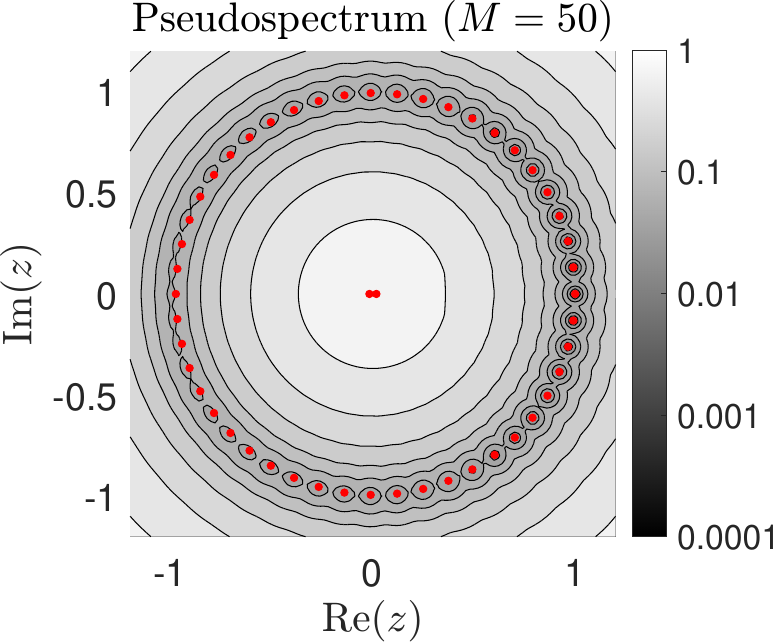}\hfill
\includegraphics[width=0.24\textwidth]{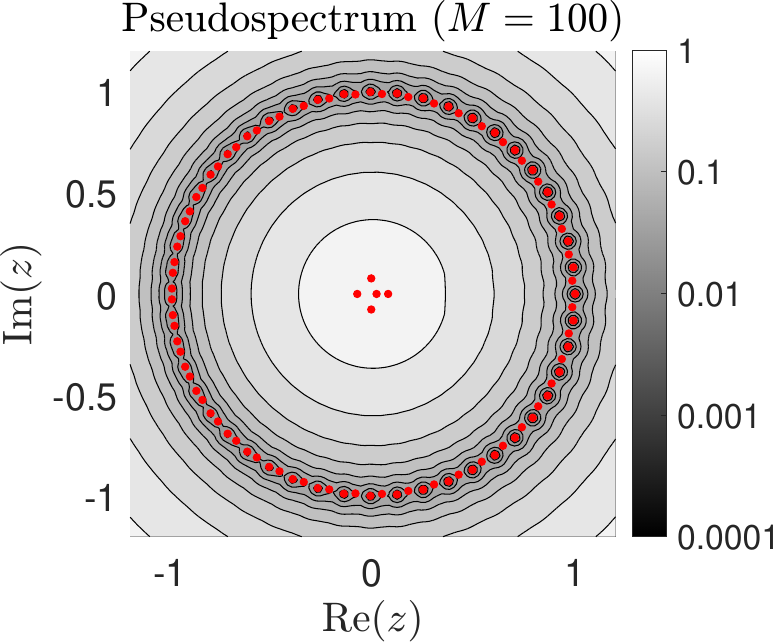}\hfill
\includegraphics[width=0.24\textwidth]{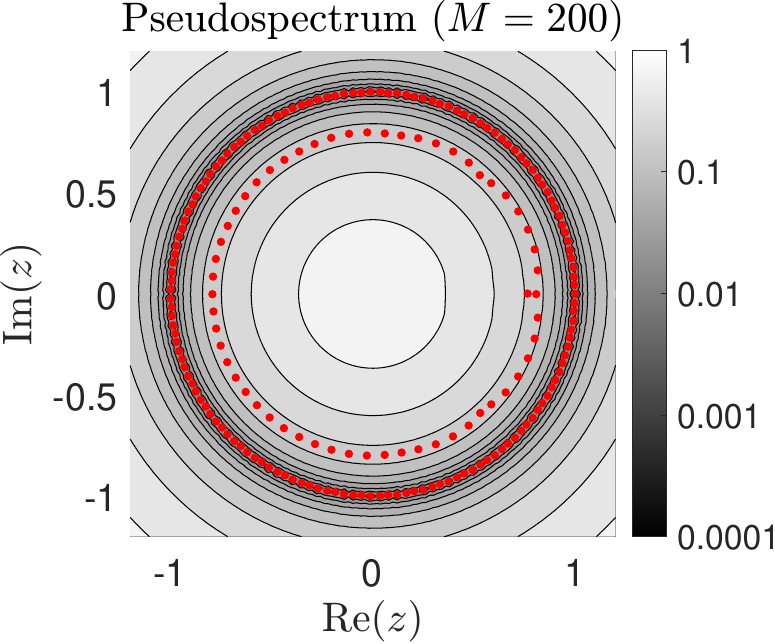}
\caption{Pseudospectra computed using \cref{alg:kernel_ResDMD2} for the cylinder flow. The shaded greyscale corresponds to the value of $\epsilon$, and the kernelized EDMD eigenvalues are shown in red, some of which are spurious. We can detect spurious eigenvalues through pseudospectra.}
\label{fig:cylinder2}
\end{figure}

\begin{figure}
\centering
\includegraphics[width=0.24\textwidth]{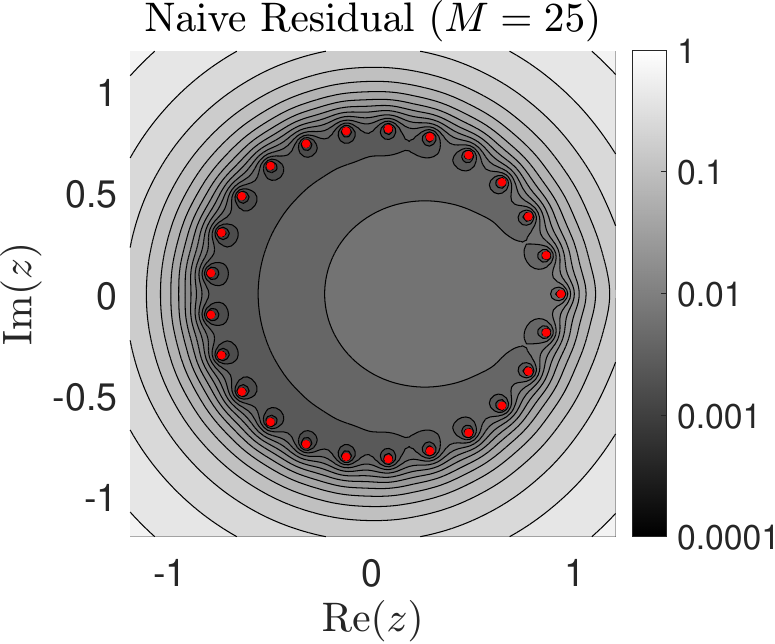}\hfill
\includegraphics[width=0.24\textwidth]{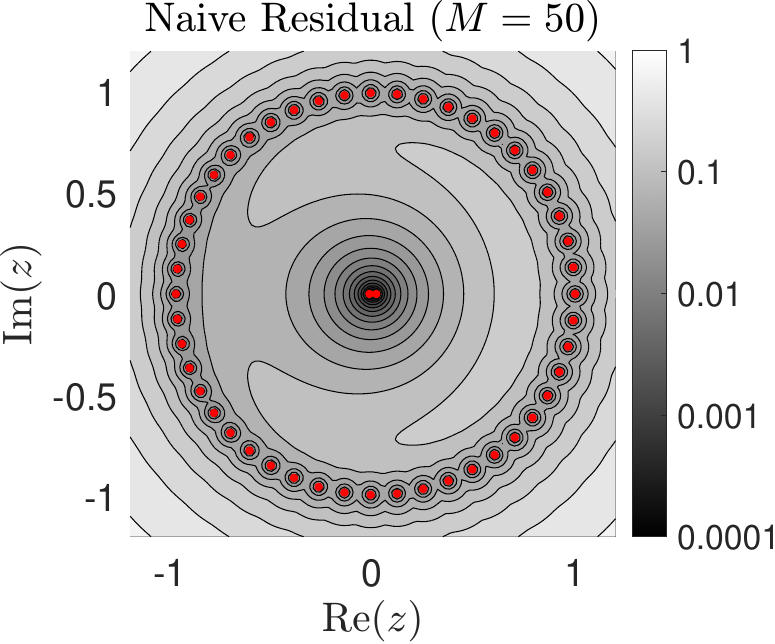}\hfill
\includegraphics[width=0.24\textwidth]{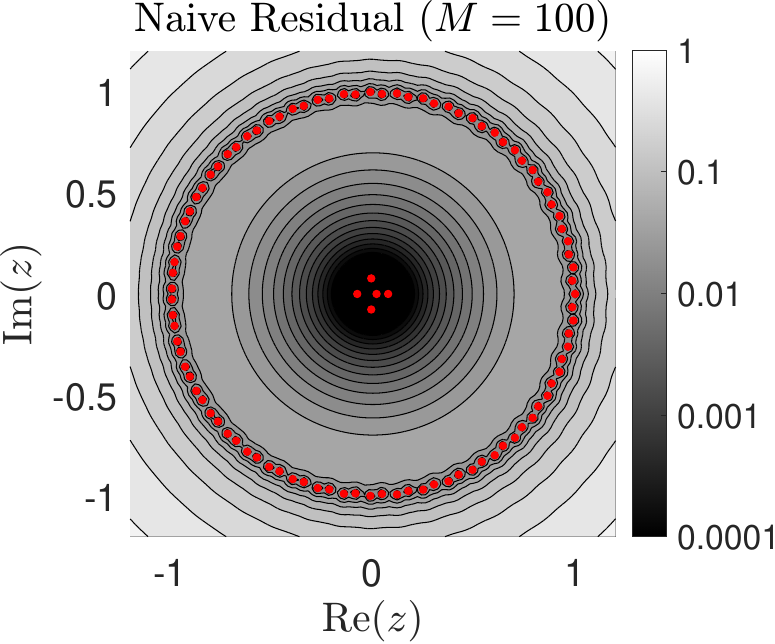}\hfill
\includegraphics[width=0.24\textwidth]{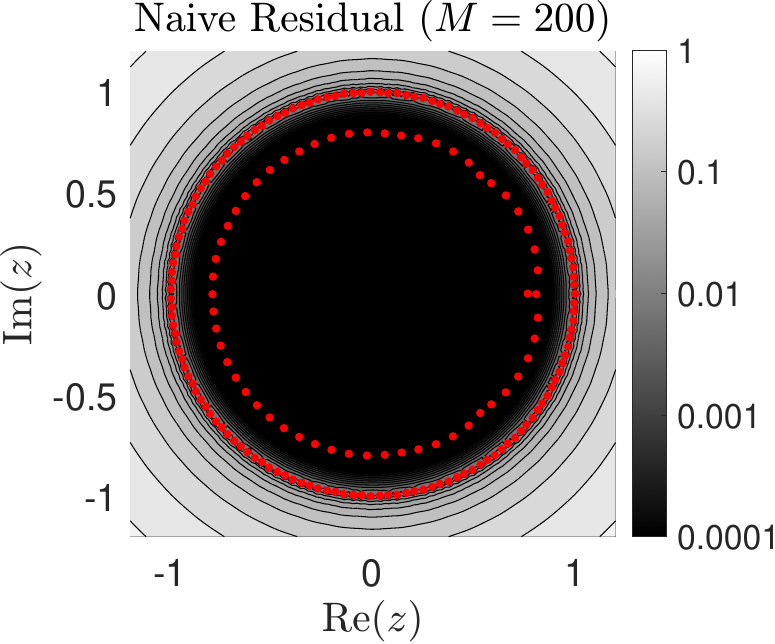}
\caption{Same as \cref{fig:cylinder2}, but now pseudospectra are computed using the naive residual in \cref{biagbfvuibi}. The approximation of the pseudospectrum is clearly wrong and we cannot detect spurious eigenvalues (without the additional knowledge that the true eigenvalues should be on the circle for this example).}
\label{fig:cylinder3}
\end{figure}

\subsection{Example 2: Verified Koopman modes of a periodic cascade of aerofoils}

\begin{figure}
\centering
\includegraphics[width=0.49\textwidth]{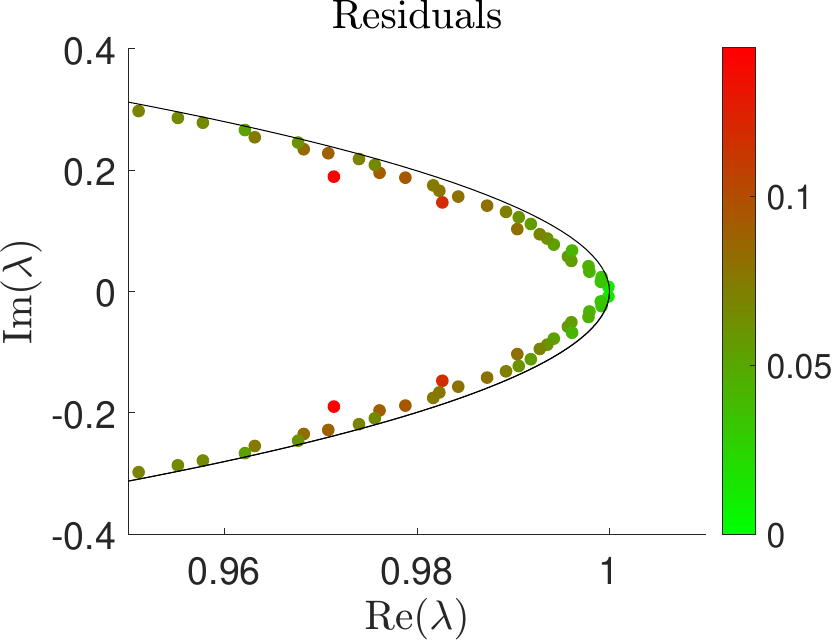}\hfill
\includegraphics[width=0.49\textwidth]{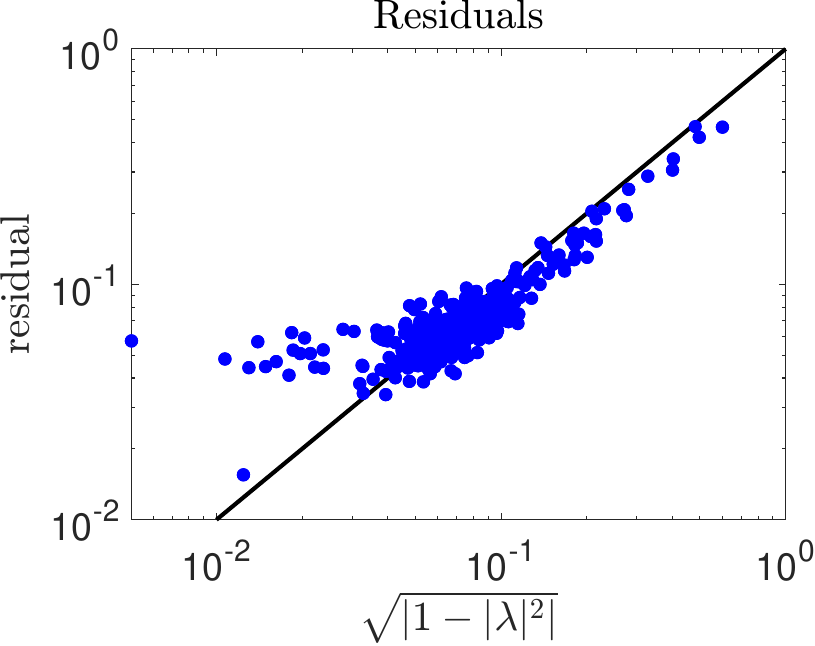}
\caption{Residuals computed using \cref{alg:exact_ResDMD1} for the periodic cascade of aerofoils. The black curve in the left panel is a portion of the unit circle.}
\label{fig:cascade_eigenvalues}
\end{figure} 

\begin{figure}
\centering
\includegraphics[width=1\textwidth]{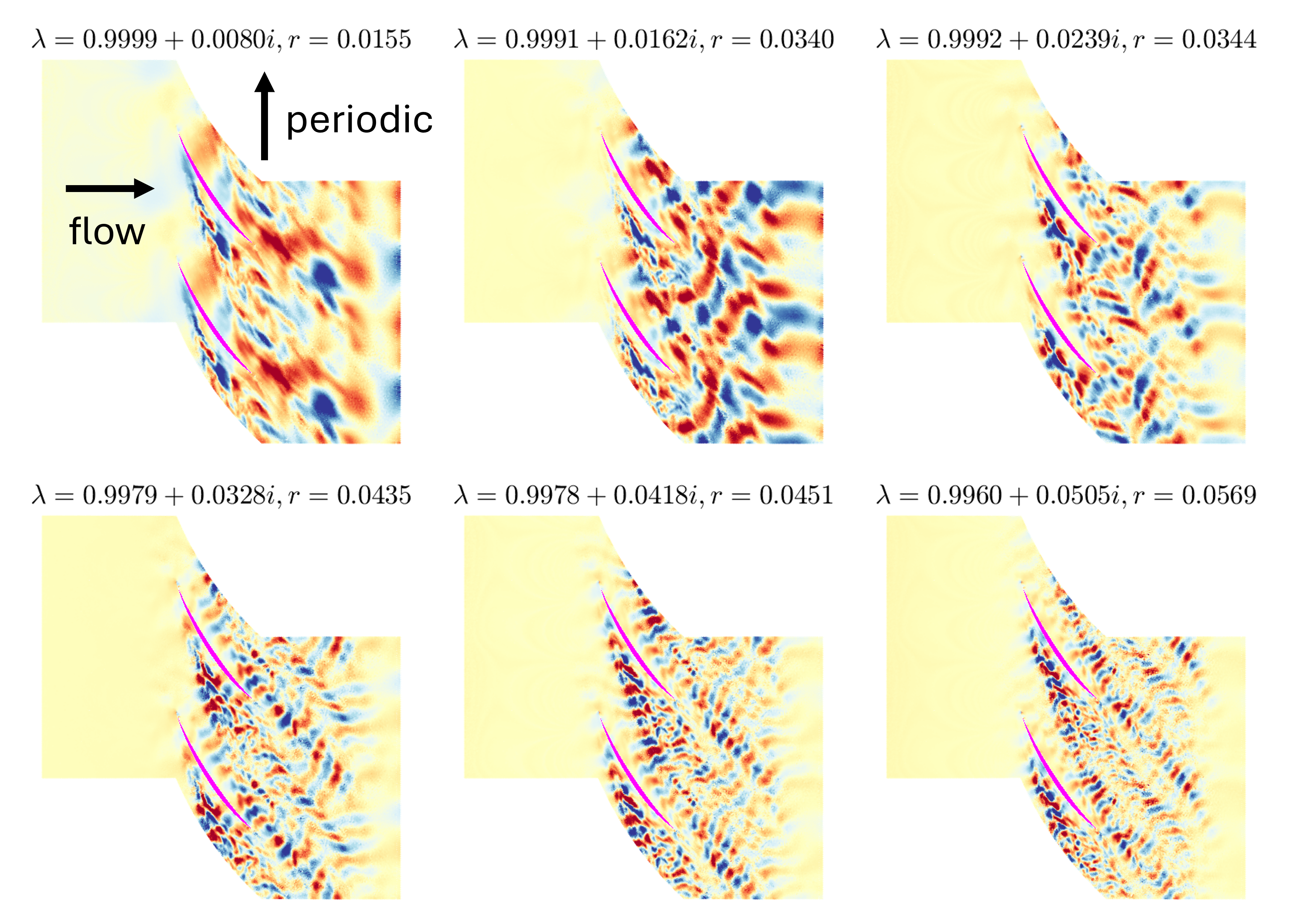}
\caption{Verified DMD modes for the periodic cascade of aerofoils. Due to conjugate symmetry, we have only shown modes with $\mathrm{Im}(\lambda)\geq 0$. The filled-in magenta regions show the positions of the aerofoils, and the setup is periodic in the vertical direction.}
\label{fig:verified_modes}
\end{figure}

Next, we consider a large-scale wall-resolved turbulent flow past a periodic cascade of aerofoils with a stagger angle $56.9^{\circ}$ and a one-sided tip gap. The setup is motivated by the need to reduce noise sources from flying objects~\cite{peake2012modern}. We use a high-fidelity simulation that solves the fully nonlinear Navier--Stokes equations~\cite{koch2021large} with Reynolds number $3.88\times10^5$ and Mach number $0.07$. The data consists of a 2D slice of the (mean-subtracted) pressure field, measured at $295,122$ points, for a single trajectory of $M=700$ snapshots sampled every $2\times 10^{-5}$s. \cref{fig:cascade_eigenvalues} shows the exact DMD eigenvalues and residuals computed using \cref{alg:exact_ResDMD1} of ResDMD. We have zoomed in close to the modes of interest near $\lambda=0$. This time, we do not see exact agreement with $\sqrt{1-|\lambda|^2}$. \cref{fig:verified_modes} shows the first six DMD modes. ResDMD allows us to compute residuals, thus providing an accuracy certificate. Similar results can be obtained using kernelized EDMD. The exact eigenvalues, modes, and residuals generally vary, depending on the chosen kernel.

\subsection{Example 3: Compression of transient shockwave}

We investigate a near-ideal acoustic monopole source in this final example and use experimental data. When a high-energy laser beam is focused on a point, the air ionizes, and plasma is generated due to the extremely high electromagnetic energy density. As a result of the sudden energy deposit, the air volume undergoes a sudden expansion that generates a shockwave.  The important acoustic characteristic is a short time period of initial supersonic propagation speed. When observed from the far field, this initial supersonic propagation is observed as a nonlinear characteristic. The data set consists of 65 realizations of laser-induced plasma sound signature data measured at a sampling rate of $f_s=$1.25 MS/s (million samples per second) \cite{szHoke2022investigating}. The data is windowed to the shockwave's support, corresponding to $123$ timesteps. We split the data into $60$ realizations to learn the kernelized EDMD matrix and five realizations to test the prediction of the Koopman operator, as shown in \cref{fig:LIP_data}. We use time delay embedding to obtain an ambient statespace dimension $d=10$. This corresponds to $M=6780$ snapshots. We apply kernelized EDMD with $r=200$ and the Gaussian kernel $\mathcal{S}(\xv,\xv')=\exp(-\|\xv-\xv'\|^2/c^2)$, where $c$ is the average $\ell_2$-norm of the mean-subtracted snapshot.

\begin{figure}
\centering
\includegraphics[width=0.5\textwidth]{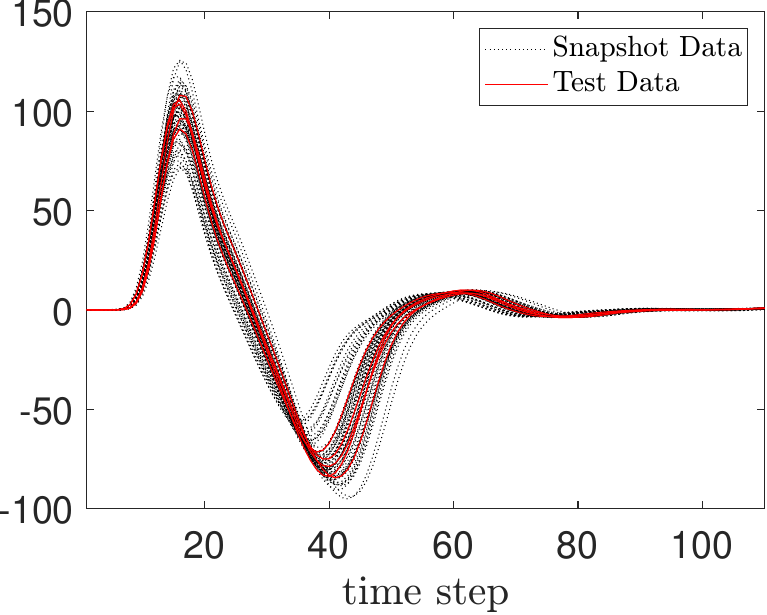}
\caption{The pressure field of the 65 realizations of the laser-induced plasma shockwave. The 6o realizations used to produce the Koopman matrix are shown in black, and the 5 randomly chosen realizations to test forecasts are shown in red.}
\label{fig:LIP_data}
\end{figure}

\begin{figure}
\centering
\includegraphics[width=0.33\textwidth]{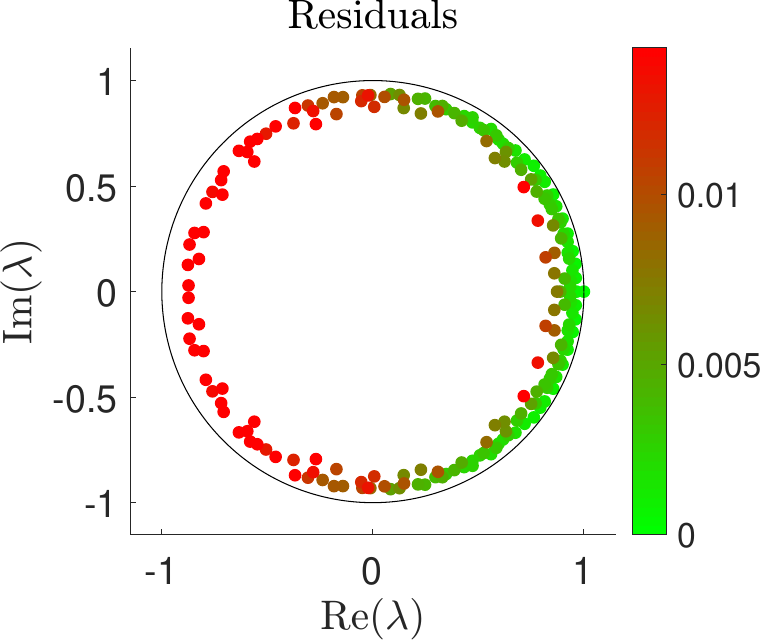}\hfill
\includegraphics[width=0.33\textwidth]{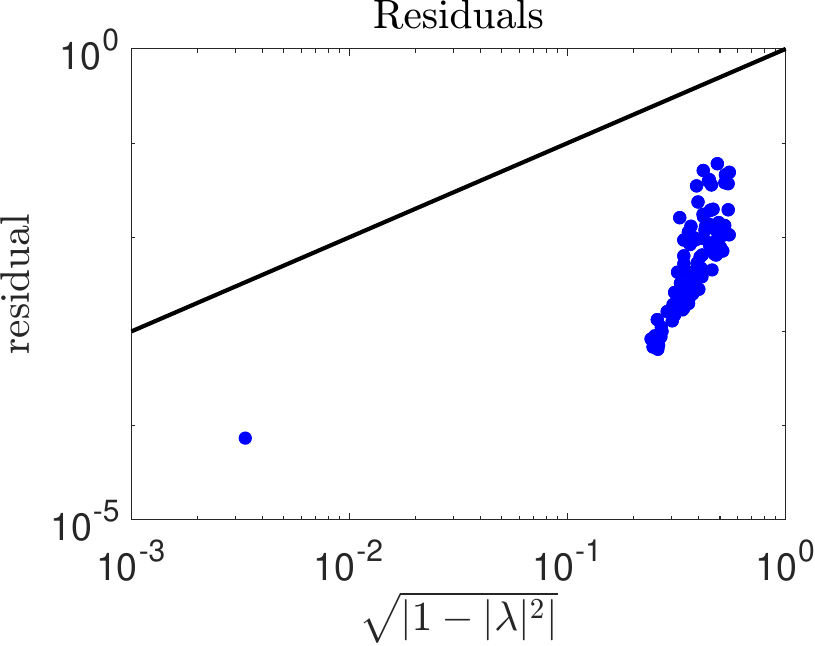}\hfill
\includegraphics[width=0.33\textwidth]{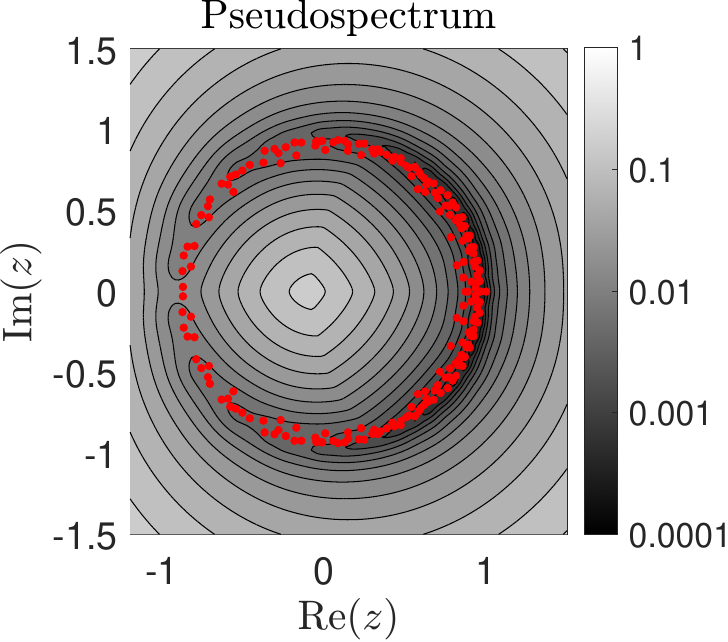}
\caption{Left and middle: The residuals of the kernelized EDMD eigenvalues computed using \cref{alg:kernel_ResDMD1} for the shockwave. Right: The pseudospectra computed using \cref{alg:kernel_ResDMD2}.}
\label{fig:LIP_spectra}
\end{figure}

\begin{figure}
\centering
\includegraphics[width=0.49\textwidth]{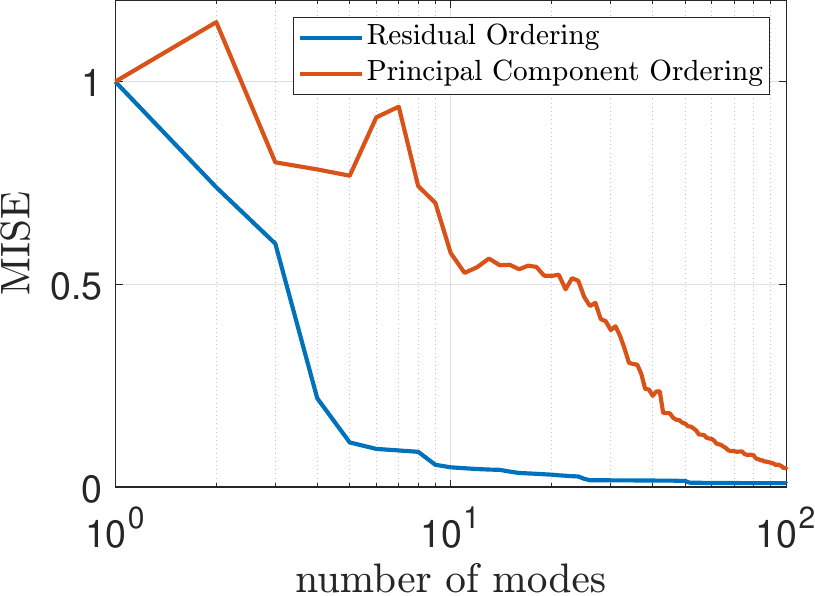}\hfill
\includegraphics[width=0.46\textwidth]{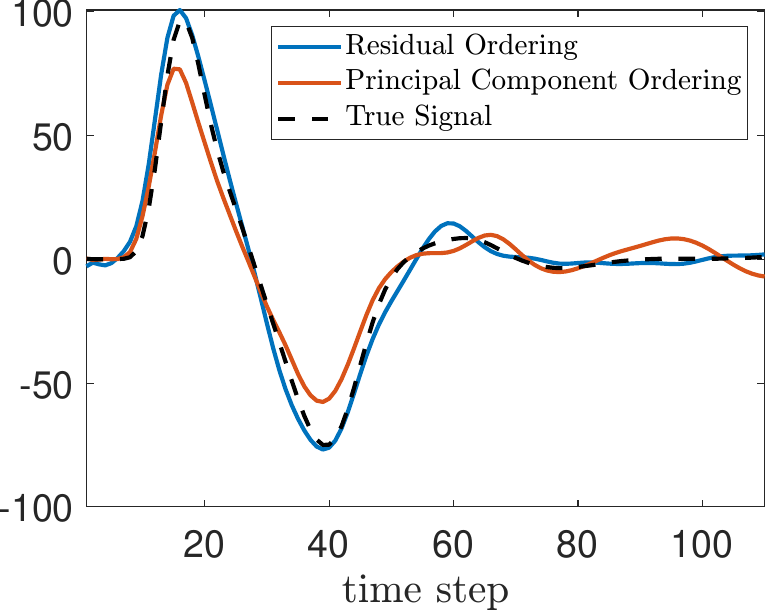}
\caption{Left: The MISE of the forecasts for the laser-induced plasma shockwave. Right: An example reconstruction using 40 modes.}
\label{fig:LIP_compresssion}
\end{figure}

\cref{fig:LIP_spectra} shows the kernelized EDMD eigenvalues and their residuals computed using \cref{alg:kernel_ResDMD1}. We have also shown the pseudospectrum computed using \cref{alg:kernel_ResDMD2}. The residuals lie far from the curve $\sqrt{1-|\lambda|^2}$ and the pseudospectrum shows a clear presence of non-normality and transient behavior. Notice also that several spurious eigenvalues lie close to $\lambda=1$.

We now show a key benefit of computing residuals. We restrict the Koopman mode decomposition to a fixed number of modes in two ways. The first way is to restrict $r$ in kernelized EDMD, corresponding to a kernel principal component ordering. The second way is to restrict the number of modes according to their residuals (keeping the first $r$ modes with the smallest residual). The IMSE of the Koopman mode decomposition of forecasts on the five training data sets are shown in \cref{fig:LIP_compresssion}, along with a typical reconstruction of both orderings when keeping only 40 modes. We see the clear benefit of the residuals in allowing a more efficient compression of the system. We can use ResDMD to compress the representation of the dynamics.


\normalsize

\begin{spacing}{.9}
\bibliographystyle{unsrt}
{\small\linespread{0.9}\selectfont{}
\bibliography{bib_file_DMD3}}
 \end{spacing}
\end{document}